\documentclass{amsart}
\usepackage{amssymb,amsmath}
\usepackage{pgf,tikz}
\usetikzlibrary{arrows}
\usepackage{hyperref}

\hypersetup{
  colorlinks  = true, 
  urlcolor    = black, 
  linkcolor   = black, 
  citecolor   = black 
}

\newcommand{\TT}{\mathbb{T}}
\newcommand{\PP}{\mathbb{P}}
\newcommand{\NN}{\mathbb{N}}
\newcommand{\I}{\mathcal{I}}
\renewcommand{\P}{\mathcal{P}}

\newcommand{\Q}{\mathcal{Q}}
\newcommand{\C}{\mathcal{C}}

\newcommand{\K}{\mathcal{K}}

\renewcommand{\t}{\mathbf{t}}

\newcommand{\ds}{\displaystyle}
\newcommand{\ts}{\textstyle}
\renewcommand{\ss}{\scriptstyle}
\newcommand{\sss}{\scriptscriptstyle}

\newcommand{\reg}{\operatorname{reg}}
\newcommand{\set}[1]{\left \{ #1 \right \}}
\newcommand{\sset}[1]{\ss\left \{ #1 \right \}\ts}

\newtheorem{theorem}{Theorem}[section]

\newtheorem{lemma}[theorem]{Lemma}
\newtheorem{cor}[theorem]{Corollary}
\newtheorem{prop}[theorem]{Proposition}
\newtheorem*{claim*}{Claim}

\theoremstyle{definition}
\newtheorem{definition}[theorem]{Definition}

\newtheorem{exam}[theorem]{Example}

\newcommand{\BDarrow}[1]{\draw[shading = axis, left color=gray, right color=gray!50!white] (0,#1) -- (.1,#1) -- (.1,#1-.2) -- 
(.2,#1-.2) -- (0,#1-.4) -- (-.2,#1-.2) -- (-.1,#1-.2) -- (-.1,#1) -- (0,#1);}

\begin{document} 
 
\title[Regularity of  vanishing ideals]%
{Regularity of the vanishing ideal \\ over a bipartite nested ear decomposition}

\author{J.~Neves}
\address{CMUC, Department of Mathematics, University of Coimbra, 3001-501 Coimbra, Portugal.}
\email{neves@mat.uc.pt}
\thanks{This work was partially supported by Centro de Matem\' atica da Universidade de Coimbra – UID/MAT/00324/2013, funded by the Portuguese Government through FCT/MEC
and co-funded by the European Regional Development Fund through the Partnership Agreement PT2020.}


\begin{abstract}
We study the Castelnuovo--Mumford regularity of the vanishing ideal over a bipartite graph 
endowed with a decomposition of its edge set. We prove that, under certain conditions, 
the regularity of the vanishing ideal 
over a bipartite graph obtained from a graph by attaching a path of length $\ell$ increases by 
$\lfloor \frac{\ell}{2}\rfloor (q-2)$, where $q$ is the order of the field of coefficients. We use this
result to show that the regularity of the vanishing ideal over a bipartite graph, $G$, endowed with a weak nested ear decomposition is equal 
to $$\ts \frac{|V_G|+ \epsilon -3}{2}(q-2),$$ where $\epsilon$ is the number of even length ears and pendant edges of the decomposition. 
As a corollary, we show that for bipartite graph, the number of even length ears 
in a nested ear decomposition starting from a vertex is constant.
\end{abstract}
\maketitle

\section{Introduction}
\label{sec: intro}
Given $G$, a simple graph, and $K$, a finite field,
$K[E_G]$ denotes the polynomial ring with coefficients in $K$, the variables 
of which are in one-to-one correspondence with the 
edge set of the graph. The vanishing ideal over $G$ is a binomial ideal of $K[E_G]$, denoted here by $I_q(G)$, 
given as the vanishing ideal of the projective toric subset parameterized by $E_G$. They were defined by Renteria, 
Simis and Villarreal in \cite{ReSiVi11}, with a view towards applications to the 
theory of linear codes and hence the presence of a finite field.  
The aim of this work is to continue the study 
of the Castelnuovo--Mumford regularity of these ideals. 
Originally this invariant is related to the error-correcting performance of the linear codes
involved, however, here, we wish to regard it strictly from the point of view 
of the link that these ideals provide between commutative algebra and graph theory. 
\smallskip

This idea has been used for many classes of ideals and the existing 
results point to interesting graph invariants. For instance, the Castelnuovo--Mumford
regularity of the \emph{edge ideal} of graph is bounded below by the induced matching number and above by 
the co-chordal cover number (cf.~\cite{HaVT08}, \cite[Lemma 2.2]{Ka06} and \cite{Wo14}). 
There are also partial results for the \emph{toric ideal} of a graph (cf.~\cite{monalg,BiO'KVT}) 
and for the \emph{binomial edge ideal} of a graph (cf.~\cite{EnZa15,KiSM16,MaMu13}).
\smallskip

The fact that, for the vanishing ideal over a graph, the quotient of $K[E_G]$ by $I_q(G)$
is a Cohen--Macaulay graded ring of dimension one explains why we know relatively more about this 
invariant in this case than in the cases of edge, toric or binomial edge ideals. 
The Castelnuovo--Mumford regularity of the vanishing ideal over a graph has been computed for many classes
of graphs, including trees, cycles (cf.~\cite{NeVPVi15}), complete graphs (cf.~\cite{GoReSa13}), 
complete bipartite graphs (cf.~\cite{GoRe08}), complete multipartite graphs (cf.~\cite{NeVP14}) 
and, more recently, parallel compositions of paths (cf.~\cite{MaNeVPVi}). Additionally we know that, in the bipartite case,
the regularities of the vanishing ideals over the members of the block decomposition of a graph completely determine the regularity 
of the vanishing ideal over the graph (see Proposition~\ref{prop: block_dec}, below).

In this work we establish a formula for the Castelnuovo--Mumford regularity 
of the vanishing ideals over graphs in the 
class of bipartite graphs endowed with certain decomposition of its edge set into paths. 
The simplest case of a such a decomposition, a so-called \emph{ear decomposition}, is a partition of the edge set of $G$ into subgraphs
$\P_0,\P_1,\dots,\P_r$ such that $\P_0$ is a vertex and, for all $1\leq i\leq r$, the path has its end-vertices in 
$\P_0\cup \cdots \cup \P_{i-1}$ and \emph{none} of its inner vertices in this union. Ear decompositions play a central role 
in graph connectivity as, by Whitney's theorem, a graph is $2$-vertex-connected if and only if it is endowed with an \emph{open} ear decomposition (one
in which every $\P_i$ with $i>1$ has distinct end-vertices). In \cite{Ep92}, Eppstein introduces the notion of \emph{nested}
ear decomposition, a special case of ear decomposition in which, firstly, the paths $\P_i$ are forced to have end-vertices in a (same)
$P_j$, for some $j<i$, and, secondly, a \emph{nesting} condition is to be satisfied for two paths having their end-vertices in a same $\P_j$ 
(see Definition~\ref{def: nested decompositions of a graph}). 
\smallskip

By Theorem~\ref{thm: regularity of a nested ear decomposition}, below, it follows that
the Castelnuovo--Mumford of the quotient of $K[E_G]$ by $I_q(G)$, when $G$ is endowed 
with a nested ear decomposition, is given by:
\begin{equation}\label{eq: the regularity for weak nested}
\ts \frac{|V_G| + \epsilon-3}{2}(q-2)
\end{equation}
where $\epsilon$ is the number of paths of even length in the decomposition and $q$ is the order of the field $K$.
As a corollary, we deduce that the number of even length paths in any nested ear decomposition  
of a bipartite graph, that starts from a vertex, is constant (cf.~Corollary~\ref{cor: number of even length paths is constant}).
\smallskip

This article is organized as follows. In the next section we set up the notation used throughout 
and define the vanishing ideal over a graph. We recall several characterizations of this ideal which 
allow a direct definition without mentioning the projective toric subset parameterized
by the edges of the graph. We also recall the Artinian reduction technique, which is the main 
tool in the computation of the regularity (cf.~Proposition~\ref{prop: computing reg by reducing to Artinian quotient}).
After reviewing some known values of the regularity (cf.~Table~\ref{table: values of regularity}) we 
go through the existing results bounding the regularity in terms of combinatorial data on the graph, among
these, the bound from the independence number of the graph (cf.~Proposition~\ref{prop: independent sets and regularity}).
Other results reviewed in this section include the upper bounds obtained from a spanning subgraph and
from an edge cover and two other results, one relating the regularity with the block decomposition and another relating
it with the leaves of the graph. Sections~\ref{sec: ears and regularity} and \ref{sec: nested ear decomposition}
contain the main results of this work. In Section~\ref{sec: ears and regularity}, we investigate the contribution
to the regularity of the vanishing ideal over a graph obtained from another graph by attaching to it a path by its end-vertices. 
Theorem~\ref{thm: adding a path to the endpoints of path with degree 2 inner vertices} states that, 
under some conditions, the regularity increases by 
$\lfloor \frac{\ell}{2}\rfloor (q-2)$,
where $\ell$ is the length of the path attached and $q$ the cardinality of $K$.
In Section~\ref{sec: nested ear decomposition}, we use the previous result to establish the 
regularity of a bipartite graph endowed with a \emph{weak} nested ear decomposition. 
This notion is a slight generalization of the notion of nested ear decomposition and 
arises naturally in the context of the proof of Theorem~\ref{thm: regularity of a nested ear decomposition}. Its
distinctive feature is that one allows the existence of pendant edges in the decomposition. Theorem~\ref{thm: regularity of a nested ear decomposition}
expresses the regu\-la\-ri\-ty of a bipartite graph endowed with a weak nested ear decomposition by the formula \eqref{eq: the regularity for weak nested}, 
where, now, $\epsilon$ is the number of even length paths and pendant edges. Corollary~\ref{cor: number of even length paths is constant},
stating that the number of even length paths in a nested ear decomposition of a graph is constant,
is then a direct consequence of this formula. As an application of Theorem~\ref{thm: regularity of a nested ear decomposition}
we finish by producing a family of graphs with regularities arbitrarily larger than the lower bound given by their
independence numbers.

\section{Preliminaries}
\label{sec: prelim}
The graphs considered in this work are assumed to be simple graphs 
(finite, undirected, loopless and without multiple edges). Additionally,
we will assume throughout that no isolated vertices occur. 
To simplify the notation, we assume that the vertex set, $V_G$, is a subset of $\NN$.

\subsection{The vanishing ideal over a graph}
We will denote by $K$ a finite field of order $q>2$.   
Given a graph $G$, we consider a polynomial ring with coefficients in $K$ 
the variables of which are in bijection with the edges of $G$ and denote it by $K[E_G]$. 
A variable in $K[E_G]$ corresponding to and edge $\set{i,j}\in E_G$ will be denoted by $t_{ij}$, which 
is the abbreviated form of $t_{\set{i,j}}$.
Given an non-negative integer valued function on the edge set, $\alpha \in \NN^{E_G}$,
the monomial $\t^\alpha\in K[E_G]$ is, by definition, 
$$
 \t^\alpha = \prod_{\sss \set{i,j}\in {E_G}}\hspace{-.3cm} t_{ij}^{\alpha\set{i,j}}.
$$
We say that $\t^\alpha$ is supported on the edges of a subgraph $H\subset G$ if 
$$\ts \alpha\set{i,j} \not = 0 \iff \set{i,j}\in E_H.$$

\noindent
Consider $\PP^{|E_G|-1}$, 
the projective space over $K$ with coordinate ring $K[E_G]$ and let $\PP^{|V_G|-1}$, be the projective space 
with coordinate ring $K[x_i : i\in V_G]$. The ring homomorphism $\varphi\colon K[E_G]\to K[x_i : i \in V_G]$
given by: 
\begin{equation}\label{eq: L231}
t_{ij}\mapsto x_ix_j
\end{equation}
defines a rational map 
$\varphi^\sharp\colon  \PP^{|V_G|-1} \to \PP^{|E_G|-1}$,  the restriction of which to the projective torus, 
$\TT^{|V_G|-1}$, the subset of projective space of points with every coordinate a nonzero scalar, 
is a regular map. 

\begin{definition}\label{def: the ideal}
The projective toric subset parameterized by $G$ is the subset of $\PP^{|E_G|-1}$ defined by:
$$
\ts X = \varphi^\sharp (\TT^{|V_G|-1}) \subset \PP^{|E_G|-1}.
$$ 
The vanishing ideal of $X$ is denoted by 
$I_q(G)\subset K[E_G]$.
\end{definition}

We note that $I_q(G)$ can be defined directly from $G$, without reference to $X$, as the ideal generated by the homogeneous
polynomials $f\in K[E_G]$ which vanish after substitution of each variable $t_{ij}$ by $a_ia_j$, for all $a_i\in K^*$, with $i\in V_G$. 
For this reason we refer to $I_q(G)$ simply as the \emph{vanishing ideal over $G$}.
\smallskip

The ideal $I_q(G)$ was defined in \cite{ReSiVi11}. 
Being a vanishing
ideal,  it is automatically a radical, graded ideal. We also know that $I_q(G)$ 
has a binomial generating set. 
The fact that $I_q(G)$ contains the vanishing ideal of the torus over the finite field $K$, which is given by
\begin{equation}\label{eq: 140}
\ts I_q = \bigl(t_{ij}^{q-1} - t_{kl}^{q-1} : \;\;\ss\set{i,j},\set{k,l} \ts \in E_G\bigr),
\end{equation}
implies 
that the height of $I_q(G)$ is $|E_G|-1$ and hence the 
quotient $K[E_G]/I_q(G)$ is a one-dimensional graded ring. 
Additionally, since any monomial in $K[E_G]$ is a regular element in this quotient (since no variable vanishes 
on the torus), we deduce that $K[E_G]/I_q(G)$ is Cohen--Macaulay. We refer the reader to \cite[Theorem~2.1]{ReSiVi11} for complete proofs of these statements.
\smallskip 

The ideal $I_q(G)$ can be related to the toric ideal of $G$, i.e., the ideal $P_G\subset K[E_G]$ given by  
$P_G = \ker \varphi$, where $\varphi \colon K[E_G]\to K[x_i : i\in V_G]$ is the map defined by \eqref{eq: L231}. 
It can be shown (see \cite[Theorem~2.5]{ReSiVi11}) that 
\begin{equation}\label{eq: L195}
\ts I_q(G) = (P_G+I_q):(\t^* )^\infty,
\end{equation}
where $I_q$ is the vanishing ideal of the torus, given in \eqref{eq: 140}, and by $\t^*$ we denote the product of all variables
of the polynomial ring $K[E_G]$,
$$
\ts\t^* = \hspace{-.2cm}\underset{\sss\set{i,j}\in E_G}{\prod}  \hspace{-.15cm} t_{ij}.
$$
The relation with the toric ideal \eqref{eq: L195} reinforces the idea, already expressed above, that $I_q(G)$ can be defined without 
any reference to the projective toric subset $X$ parameterized by $E_G$. Yet another way to do this is by a characterization of the 
set homogeneous binomials of $I_q(G)$, achieved by the 
following proposition. The proof of this result can be found in \cite[Lemma~2.3]{NeVP14}.

\begin{prop}\label{prop: congruences}
Let $\t^\nu - \t^\mu \in K[E_G]$ be a homogeneous  binomial. Then 
$\t^\nu -\t^\mu$ belongs to  $I_q(G)$ if and only if, for all 
$i\in V_G$, 
\begin{equation}\label{eq: 181}
\ds\sum_{k\in N_G(i)} \nu \sset{i,k} \ds \equiv \sum_{k\in N_G(i)} \mu \sset{i,k} \pmod{q-1},
\end{equation}
where $N_G(\cdot)$ denotes the set of neighbors of a vertex.
\end{prop}

With this characterization of $I_q(G)$ by means of a generating set of homogeneous binomials satisfying \eqref{eq: 181}, 
the following relation between the ideal $I_q(G)$ and the vanishing ideal over a subgraph of $G$ is easy to prove.

\begin{cor}\label{cor: inclusion of vanishing ideals}
Let $H$ be a subgraph of $G$. Then, under the inclusion of polynomial rings \mbox{$K[E_H]\subset K[E_G]$,}  
we have $I_q(H) =  I_q(G)\cap K[E_H]$.
\end{cor}

Despite the multiple characterizations of $I_q(G)$, a complete classification of the subgraphs of $G$ that support binomials of a 
minimal binomial generating set of $I_q(G)$ is still lacking, for general $G$; in contrast with the case of the toric ideal
$P_G$ in which the binomials in a minimal generating set are in one-to-one correspondence with the closed even walks 
on the graph.

\subsection{Castelnuovo--Mumford regularity}
Recall that if $S$ is a polynomial ring and $M$ is any graded $S$-module, the Castelnuovo--Mumford regularity
of $M$ is, by definition, 
$$
\reg M = \max_{i,j} \set{j-i : \beta_{ij} \not = 0},
$$
where $\beta_{ij}$ are the graded Betti numbers of $M$. The Castelnuovo--Mumford regularity of $K[E_G]/I_q(G)$ is 
thus an integer we can associate to any simple graph without isolated vertices. 

\begin{definition}
Let $G$ be a simple graph without isolated vertices and $K$ a finite field. 
We define the Castel\-nuo\-vo--Mumford regularity of $G$ over the field $K$
to be the regularity of the quotient $K[E_G]/I_q(G)$ and we denote it by $\reg G$.
\end{definition}

Since $K[E_G]/I_q(G)$ is a Cohen--Macaulay one-dimensional
graded ring, its re\-gu\-larity coincides with its 
\emph{index of regularity}, i.e., the least integer from which the value of the Hilbert function equals the value of the 
Hilbert Polynomial (cf.~\cite[Proposition 4.2.3]{monalg}). 
Additionally, given that any monomial $\t^\delta \in K[E_G]$ is a re\-gu\-lar element of 
$K[E_G]/I_q(G)$ and the quotient of $K[E_G]$ by the extended ideal, $(I_q(G),\t^\delta)$, is
a zero-dimensional graded ring with index of regularity equal to 
$\reg G + \deg \t^\delta$, we get:
$$
\reg G = \min\set{i :  \dim_K \bigl ( K[E_G]/(I_q(G),\t^\delta) \bigr)_i = 0} - \deg \t^\delta.
$$
\smallskip

The idea of taking the Artinian quotient $K[E_G]/(I_q(G),\t^\delta)$ to compute $\reg G$ is the main ingredient
in the proof of the next proposition, which will be used several times in this article. 
See \cite[Propositions~2.2 and 2.3]{MaNeVPVi} for a proof.

\begin{prop}\label{prop: computing reg by reducing to Artinian quotient}
Let $G$ be a graph, $\t^{\delta}\in K[E_G]$ a monomial and $d$ a positive integer. Then, 
$$
\ts \reg G \leq d - \deg(\t^\delta)
$$ 
if and only if  for every monomial 
$\t^{\nu}$ of degree $d$ there exists $\t^\mu$, of degree $d$, such that $\t^\delta\mid \t^\mu$ and $\t^\nu-\t^\mu \in I_q(G)$.
\end{prop}
 
\subsection{Graph invariants and the regularity} 
Table~\ref{table: values of regularity} contains the values of the Castelnuovo--Mumford regularity of several families of graphs. 
The simplest cases are those of a tree and an odd cycle. 
\begin{table}[ht]\renewcommand{\arraystretch}{2}
\begin{tabular}{ll}
Graph & $\reg G$ \\ \hline 
Tree with $s$ edges & $(|V_G|-2)(q-2)$\\ \hline 
Odd cycle & $(|V_G|-1)(q-2)$\\ \hline 
Even cycle  & $\frac{|V_G|-2}{2}(q-2)$\\ \hline 
Complete graph $\K_n$, $n\geq 4$  & $\lceil (n-1)(q-2)/2 \rceil$\\ \hline 
Complete bipartite graph $\K_{a,b}$ & $(\max\set{a,b} -1 )(q-2)$ \\ \hline
\\
\end{tabular}
\caption{}
\label{table: values of regularity}
\end{table}
These are the simplest cases because, in both, $X$, 
the projective toric subset parameterized by $E_G$, coincides with the torus (cf.~\cite[Corollary~3.8]{ReSiVi11}) and therefore 
the ideal $I_q(G)$ is equal to the vanishing ideal of the torus, $I_q$, given in (\ref{eq: 140}). The fact
that $I_q$ is a complete intersection enables the straightforward computation of the regularity.
In \cite{NeVPVi15} the case of an even cycle was dealt with. The cases of the complete graph and of the complete bipartite graph were studied in 
\cite{GoReSa13} and \cite{GoRe08}, respectively. 
\smallskip

By now, there are many ways to produce estimates for the Castelnuovo--Mumford regularity of a particular graph using 
combinatorial invariants of the graph.  We begin by mentioning the lower bound obtained from the vertex independence number of the graph.

\begin{prop}[{\cite[Proposition~2.7]{MaNeVPVi}}]\label{prop: independent sets and regularity}
If $V\subset V_G$ is a set of $r$ independent vertices, such that 
the edge set of $G-V$ is not empty, then \mbox{$\reg G \geq r(q-2)$.}
\end{prop}

\noindent
Since $G$ has no isolated vertices, it follows from this result that 
\begin{equation}\label{eq: 319}
\reg G \geq (\alpha(G)-1)(q-2),
\end{equation}
where $\alpha(G)$ is the vertex independence number of $G$. 
However, as can be easily seen by the values of the regularity of Table~\ref{table: values of regularity},
this bound is not sharp if $G$ is non-bipartite or, even for a bipartite graph, if it fails to be $2$-connected.
As an application of Theorem~\ref{thm: regularity of a nested ear decomposition},
we shall give an infinite family of $2$-connected bipartite graphs 
for which the bound \eqref{eq: 319} is not sharp. (See Example~\ref{exam: taino sun}.)
\smallskip

The operation of vertex identification also yields lower bounds for the Castel\-nuovo--Mumford regularity of $G$. The next result was proved in \cite[Proposition~2.5]{MaNeVPVi}.

\begin{prop}\label{prop: vertex identification}
Let $v_1$ and $v_2$ be two nonadjacent vertices of $G$ and let $H$ be the simple graph obtained after  
identifying $v_1$ with $v_2$. Then $\reg G \geq \reg H$.	
\end{prop}

Note that the identification of two vertices can create multiple edges. By simple graph, in the statement, we refer to the graph obtained after the removal 
of all multiple edges created.
\smallskip

Bounds for the Castelnuovo--Mumford regularity of a graph can also be obtained from its subgraphs. 
The next result follows from \cite[Lemma 2.13]{VPVi13}.

\begin{prop}\label{prop: bound from a spanning subgraph}
Let $H$ be a spanning subgraph of $G$ which is non-bipartite if $G$ is non-bipartite. Then
$\reg G\leq \reg H$.
\end{prop}

\noindent
Reversing the roles of $G$ and $H$, this result can also be used to produce 
lower bounds of the regularity. For instance, if $G$ is bipartite and spans a 
$\K_{a,b}$ then 
$$
\ts \reg G \geq  (\max\set{a,b}-1)(q-2)
$$
and if $G$ is non-bipartite with $|V_G|\geq 4$, then
$$
\ts \reg G \geq \reg \K_{|V_G|} =  \bigl \lceil\frac{(|V_G|-1)(q-2)}{2} \bigr	\rceil.
$$

Another way to obtain upper bounds for the regularity of a graph is by using a decomposition
of $G$ into two subgraphs with, at least, one edge in common. 

\begin{prop}[{\cite[Proposition~2.6]{MaNeVPVi}}]\label{prop: reg bound from graph dec}
If $H_1$ and $H_2$ are two subgraphs of $G$ with a common edge and 
$G=H_1\cup H_2$ then 
$$
\reg G \leq \reg H_1 + \reg H_2.
$$
\end{prop}

A graph is said $2$-vertex-connected (or simply $2$-connected) if
$|V_G|\geq 3$ and $G-v$ is connected for every $v\in V_G$. Any 
graph can be decomposed into a set of edge disjoint subgraphs consisting of either isolated vertices, 
single edges (called bridges) or maximal $2$-connected subgraphs. 
This decomposition is called the block decomposition 
of the graph.
In \cite{NeVPVi14} the relation between the regularity of a bipartite graph and the 
regularities of the members of its block decomposition 
was described. 

\begin{prop}[{\cite[Theorem~7.4]{NeVPVi14}}]\label{prop: block_dec}
Let $G$ be a simple bipartite graph with\-out isolated vertices and let $G = H_1\cup \dots \cup H_m$ be the block decomposition of $G$, then
\begin{equation}\label{eq: 315}
\ts\reg G = \sum_{k=1}^m \reg H_i + (m-1)(q-2).
\end{equation}
\end{prop}

The previous result does not hold if we drop the bipartite assumption. 
The graph in Figure~\ref{fig: block dec requires bipartite} is a counterexample. 
We used \emph{Macaulay2}, \cite{M2}, to compute 
its Castelnuovo--Mumford for some values of  the order of the ground field. For $q\in \set{3,4,5,7,8,9,11,13,16}$, the regularity is given by the formula $\lceil 5(q-2)/2 \rceil$. 
On the other hand, its block decomposition has three blocks; two triangles, of regularity $2(q-2)$, and a cut-edge, of
regularity zero. Using this in formula (\ref{eq: 315}) yields $6(q-2)$.

\begin{figure}[ht]
\begin{center}
\begin{tikzpicture}[line cap=round,line join=round, scale=1.5]
\draw [fill=black] (-.5,0) circle (1pt);
\draw [fill=black] (.5,0) circle (1pt);
\draw [fill=black] (-1.25,.5) circle (1pt);
\draw [fill=black] (-1.25,-.5) circle (1pt);
\draw [fill=black] (1.25,.5) circle (1pt);
\draw [fill=black] (1.25,-.5) circle (1pt);
\draw (-1.25,.5)-- (-1.25,-.5);
\draw (1.25,.5) -- (1.25,-.5);
\draw (-1.25,.5) --(-.5,0);
\draw (-1.25,-.5) --(-.5,0);
\draw (1.25,.5) --(.5,0);
\draw (1.25,-.5) --(.5,0);
\draw (-.5,0) -- (.5,0);
\end{tikzpicture}
\end{center}
\caption{}
\label{fig: block dec requires bipartite} 
\end{figure}

The next proposition gives an additive formula for the regularity of 
$G$ with respect to its leaves which holds for both bipartite and non-bipartite graphs.

\begin{prop}[{\cite[Proposition~2.4]{MaNeVPVi}}]\label{prop: additivity of reg on leaves}
If $v_1,\dots,v_r$ are vertices of degree one and $G^\flat$ is the graph defined by 
$G^\flat = G - \set{v_1,\dots,v_r}$ then 
$$
\reg G = \reg G^\flat + r(q-2).
$$
\end{prop}

Proposition~\ref{prop: block_dec} motivates the study of the regularity of a general $2$-connected bipartite graph. 
In view of Whitney's structure theorem for $2$-connected graphs (see Section~\ref{sec: nested ear decomposition}) 
one is naturally drawn to the problem of assessing the change produced in the regularity when we attach a path to a graph. 
We will explore this idea in the next two sections.

\section{Ears and Regularity}
\label{sec: ears and regularity}

The aim of this section is to provide a relation between the Castelnuovo--Mumford regularities 
of a graph and of the graph obtained by attaching a path by its end-vertices.
The main theorem of this section, Theorem~\ref{thm: adding a path to the endpoints of path with degree 2 inner vertices},
states that the addition of such a path increases the regularity by $\lfloor \frac{\ell}{2}\bigr \rfloor (q-2)$, 
where $\ell$ is the length of the path. In this result,
we assume that $G$ is bipartite and that the end-vertices of the path are identified with two vertices of the graph 
which, in turn, are connected in the graph by a path the inner vertices of which have degree two. Both assumptions
are necessary (see Examples~\ref{exam: ear is needed} and \ref{exam: bipartite is needed}). 
Proposition~\ref{prop: adding a path to the endpoints of an edge} addresses a special case in which we can afford to drop 
the bipartite assumption. 
\smallskip

By a path, $\P\subset G$, we mean a subgraph of $G$ endowed with an order of its vertices,
$v_0,v_1,\dots,v_\ell$, where $\ell>0$, such that $v_1,\dots,v_\ell$ are $\ell$ distinct vertices 
and $E_\P$ consists of the $\ell$ distinct edges $\set{v_i,v_{i+1}}$, $i=0,\dots,\ell-1$.
If $v_0=v_\ell$, $\P$ is also called a cycle. However, note that  
since we are assuming that the edges $\set{v_i,v_{i+1}}$ are distinct, the case $\ell = 2$ and $v_0=v_2$ is not allowed.
The \emph{inner vertices} of $\P$ are ${v_1,\dots,v_{\ell-1}}$ and the \emph{end-vertices} of $\P$ are $v_0$ and $v_\ell$. 
The set of inner vertices of $\P$ will be denoted by $\P^\circ \subset V_G$.
The number of edges in $\P$ is called the length of $\P$ and will be denoted by $\ell(\P)$.

\begin{definition}\label{def: ear}
A path $\P\subset G$ is called an \emph{ear} of $G$ if
all inner vertices of $\P$  have degree two in $G$. 
If the end-vertices of $\P$ are distinct, $\P$ is called
an open ear if they coincide, $\P$ is called a pending cycle.
\end{definition}

\begin{prop}\label{prop: adding a path to the endpoints of an edge}
Let $\P\subset G$ be an ear of $G$, of length $\ell>1$. Assume either: 
\begin{enumerate}
\item $\ell$ is odd and the end-vertices of $\P$ are distinct and adjacent in $G$ or
\item $\ell$ is even and the end-vertices of $\P$ coincide.
\end{enumerate}
Denote the graph $G-\P^\circ$ by $G^\flat$ and assume that $G^\flat$ has no isolated vertices. Then
\begin{equation}\label{eq: L583}
\textstyle \reg G = \reg (G^\flat) + \bigl \lfloor \frac{\ell}{2}\bigr \rfloor (q-2).
\end{equation}
\end{prop}

\begin{proof} Note that since $\ell>1$ and, when $\ell$ is even $\ell \geq 4$, we get $\ell \geq 3$.
Without loss of generality, we may assume that $\P$ is the path in $G$ given by  
$(1,\dots,\ell+1)$, if $\ell$ is odd or $(1,\dots,\ell,1)$ if $\ell$ is even.
(See Figure~\ref{fig: adding a path to an edge}.)
\begin{figure}[ht]
\begin{center}
\begin{tikzpicture}[line cap=round,line join=round, scale=15/13]

\draw (.5,0) -- (-.5,0);
\draw (.5,0) -- (0.75,-.01);
\draw (0.9,-.02) node {$\ss \cdots$};
\draw (.5,0) -- (.45,-.15);
\draw (-.5,0) -- (-0.7,-.01);
\draw (-0.8,-.02) node {$\ss \cdots$};
\draw (-.5,0) -- (-.45,-.15);
\draw (-.5,0) -- (-.35,-.1);

\draw [fill=black] (0.5,0) node[anchor=north west] {$\scriptstyle \ell +1$} circle (1.3pt);
\draw [fill=black] (-0.5,0) node[anchor=north east] {$\scriptstyle 1$} circle (1.3pt);
\draw (-.5,0) .. controls (-2.5,2) and (2.5,2) .. (.5,0);
\draw [fill=white,white] (0,1.5) circle (7pt);
\draw (0.02,1.48) node {$\cdots$};
\draw [fill=black] (-.91,.55) node[anchor=east] {$\scriptstyle 2$} circle (1.3pt);
\draw [fill=black] (.91,.55) node[anchor=west] {$\scriptstyle \ell$} circle (1.3pt);
\draw [fill=black] (-.7,1.32) node[anchor=south east] {$\scriptstyle 3$} circle (1.3pt);
\draw [fill=black] (.7,1.32) node[anchor=south west] {$\scriptstyle \ell-1$} circle (1.3pt);
\draw [gray,->] (-1.4,1) -- (-1.1,.9);
\draw (-1.35,.9) node[black, anchor = south east] {$\P$};

\draw (3,-.7) node {\footnotesize{(b)}};


\draw (3.2,0) -- (3,0);
\draw (3.35,-.01) node {$\ss\cdots$};
\draw (3,0) -- (2.95,-.15);
\draw (3,0) -- (2.75,-.01);
\draw (2.65,-.02) node {$\ss \cdots$};

\draw [fill=black] (3,0) node[anchor=north west] {$\scriptstyle 1$} circle (1.3pt);

\draw (3,0) .. controls (0,2) and (6,2) .. (3,0);
\draw [fill=white,white] (3,1.5) circle (7pt);

\draw (3.02,1.48) node {$\cdots$};
\draw [fill=black] (2.35,.55) node[anchor=east] {$\scriptstyle 2$} circle (1.3pt);

\draw [fill=black] (3.65,.55) node[anchor=west] {$\scriptstyle \ell$} circle (1.3pt);
\draw [fill=black] (2.3,1.32) node[anchor=south east] {$\scriptstyle 3$} circle (1.3pt);
\draw [fill=black] (3.7,1.32) node[anchor=south west] {$\scriptstyle \ell-1$} circle (1.3pt);
\draw [gray,->]  (4.3,1) -- (4,.9) ;
\draw (4.3,.9) node[black, anchor = south west] {$\P$};

\draw (.05,-.7) node {\footnotesize{(a)}};

\end{tikzpicture}
\end{center}
\caption{}
\label{fig: adding a path to an edge} 
\end{figure}
In both cases, it follows that $G$ contains an even cycle the vertex set of which coincides 
with $V_\P$. Since the generators of $P_G$, the toric ideal of $G$, are given by the closed even 
walks on $G$, using the relation between $P_G$ and $I_q(G)$, expressed in \eqref{eq: L195}, it follows that
\begin{equation}\label{eq: L654}
\renewcommand{\arraystretch}{1.5}
\left \{
\begin{array}{l}
t_{12}t_{34}\cdots t_{\ell(\ell+1)} -t_{23}t_{45}\cdots t_{(\ell+1)1}\in I_q(G),\quad \text{if $\ell$ is odd, or}\\
t_{12}t_{34}\cdots t_{(\ell-1)\ell} -t_{23}t_{45}\cdots t_{\ell 1}\in I_q(G),\quad \text{if $\ell$ is even.}\\
\end{array} 
\right.
\end{equation}
Let us start by showing that 
$$
\ts \reg G \geq \reg G^\flat + \bigl \lfloor \frac{\ell}{2}\bigr \rfloor (q-2)
$$
Fix $t_{kl}\in E_{G^\flat}$. By Proposition~\ref{prop: computing reg by reducing to Artinian quotient}, 
applied to the graph $G^\flat$,
we deduce that there exists $\t^\alpha \in K[E_{G^\flat}]$, of degree $\reg({G^\flat})$, 
for which no monomial $\t^\beta \in K[E_{G^\flat}]$ divisible by $t_{kl}$ is such that 
$\t^\alpha -\t^\beta\in I_q({G^\flat})$. 
\smallskip

\noindent
Let $\t^\nu \in K[E_{G}]$ be the monomial of degree $\reg({G^\flat})+\bigl \lfloor \frac{\ell}{2}\bigr \rfloor (q-2)$ given by:
$$
\renewcommand{\arraystretch}{1.5}
\left \{
\begin{array}{l}
\t^\nu = \t^\alpha (t_{23}t_{45}\cdots t_{(\ell-1)\ell})^{q-2},\quad \text{if $\ell$ is odd, or}\\
\t^\nu = \t^\alpha (t_{23}t_{45}\cdots t_{\ell 1})^{q-2},\quad \text{if $\ell$ is even.}\\
\end{array} 
\right.
$$
Suppose there
exists $\t^\mu \in K[E_{G}]$, with $\mu\sset{k,l}>0$, such that 
\begin{equation} \label{eq:204}
\t^\nu - \t^\mu \in I_q(G).
\end{equation}
Modifying appropriately, with the use of $t_{ij}^{q-1} - t_{kl}^{q-1} \in I_q(G)$, we may assume
that $0\leq \mu\sset{i, i+1} \leq q-2$, for all $i=1,\dots,\ell-1$, and $0\leq \mu\sset{\ell,\ell+1}\leq q-2$, if $\ell$
is odd, or \mbox{$0\leq \mu\sset{1,\ell} \leq q-2$}, if $\ell$ is even. In other words, we
may assume that the variables along the path $\P$ appear in $\t^\mu$ raised to powers not greater than $q-2$.
Then, evaluating the congruences of Proposition~\ref{prop: congruences} at the vertices $2,\dots,\ell$ we get, if $\ell$ is odd, 
$$
\ts \mu\sset{i-1,i}+\mu\sset{i,i+1} \equiv q-2,\;\forall_{i\in \set{2,\dots,\ell}}
$$
or, if $\ell$ is even, 
$$
\ts \mu\sset{i-1,i}+\mu\sset{i,i+1} \equiv q-2,\;\forall_{i\in \set{2,\dots,\ell-1}}, 
\quad \text{and}\quad  \mu\sset{\ell-1,\ell}+\mu\sset{\ell,1} \equiv q-2.
$$
where all congruences are modulo $q-1$. We deduce that there exist 
$$
a,b\in \set{0,\dots,q-2},
$$ 
with $a+b\equiv q-2$ such that, if $\ell$ is odd,
$$
\begin{cases}
\mu\sset{1,2} = \mu\sset{3,4} = \cdots = \mu\sset{\ell,\ell+1}  = a\\
\mu\sset{2,3} = \mu\sset{4,5} = \cdots = \mu\sset{\ell-1,\ell} = b
\end{cases}
$$
or, if $\ell$ is even,
$$
\begin{cases}
\mu\sset{1,2} = \mu\sset{3,4} = \cdots = \mu\sset{\ell-1,\ell}  = a\\
\mu\sset{2,3} = \mu\sset{4,5} = \cdots = \mu\sset{\ell,1} = b.
\end{cases}
$$
\smallskip

\noindent
Let $\t^\delta \in K[E_{G^\flat}]$ be the monomial
supported on ${G^\flat}$, given by 
$$
\renewcommand{\arraystretch}{1.5}
\left \{
\begin{array}{l}
\t^\mu = (t_{12}t_{34}\cdots t_{\ell(\ell+1)})^{a} (t_{23}t_{45}\cdots t_{(\ell-1)\ell })^{b}\,  \t^\delta,\quad \text{if $\ell$ is odd, or}\\
\t^\mu = (t_{12}t_{34}\cdots t_{(\ell-1)\ell})^{a} (t_{23}t_{45}\cdots t_{\ell 1 })^{b}\,  \t^\delta,\quad \text{if $\ell$ is even.}\\
\end{array} 
\right.
$$
In view of \eqref{eq: L654}, we deduce that there exists 
$\t^{\beta}\in K[E_{{G^\flat}}]$ such that 
$$
\beta\sset{k,l}\geq \mu\sset{k,l}>0
$$ 
and
\begin{equation}\label{eq:229}
\renewcommand{\arraystretch}{1.5}
\left \{
\begin{array}{l}
\t^\mu - (t_{23}t_{45}\cdots t_{(\ell-1)\ell})^{q-2}\,\t^\beta \in I_q(G),\quad \text{if $\ell$ is odd, or}\\
\t^\mu - (t_{23}t_{45}\cdots t_{\ell 1})^{q-2}\,\t^\beta \in I_q(G),\quad \text{if $\ell$ is even.}\\
\end{array} 
\right.
\end{equation}
Note that if $a+b>q-2$ then, using $t_{ij}^{q-1}-t_{kl}^{q-1}\in I_q(G)$, the powers 
of the variables $t_{23}$, $t_{45},\dots,t_{(\ell-1)\ell}$ can be reduced to $q-2$.
Combining (\ref{eq:204}) and (\ref{eq:229}) we obtain 
$$
\renewcommand{\arraystretch}{1.5}
\left \{
\begin{array}{l}
\t^\alpha (t_{23}t_{45}\cdots t_{(\ell-1)\ell})^{q-2} - (t_{23}t_{45}\cdots t_{(\ell-1)\ell})^{q-2}\,\t^\beta \in I_q(G),\quad \text{if $\ell$ is odd, or}\\
\t^\alpha (t_{23}t_{45}\cdots t_{\ell 1})^{q-2} - (t_{23}t_{45}\cdots t_{\ell 1})^{q-2}\,\t^\beta \in I_q(G),\quad \text{if $\ell$ is even.}\\
\end{array} 
\right.
$$
Since any product of variables is regular in $K[E_G]/I_q(G)$, we get
$$
\t^\alpha  - \t^\beta \in I_q(G),
$$
where, recall $\beta\sset{k,l}>0$. But this binomial, being supported on ${G^\flat}$, also belongs to $I_q({G^\flat})$. 
This contradicts the assumptions on $\t^\alpha$.
Therefore, by Proposition~\ref{prop: computing reg by reducing to Artinian quotient},
$$
\ts\reg G \geq \deg(\t^\nu) = \reg({G^\flat})+\bigl \lfloor \frac{\ell}{2}\bigr \rfloor (q-2).
$$

To prove the opposite inequality we will use Proposition~\ref{prop: reg bound from graph dec}.
If $\ell$ is odd (see Figure~\ref{fig: adding a path to an edge}a), consi\-de\-r the decomposition of $G$ given by  
$\P\cup \set{{1,\ell +1}}$ and $G^\flat$. Then,
$$
\textstyle \reg G \leq \reg G^\flat + \reg(\P\cup \set{1,\ell+1})  =   \reg G^\flat + \bigl \lfloor \frac{\ell}{2}\bigr \rfloor (q-2).
$$
If $\ell$ is even (see Figure~\ref{fig: adding a path to an edge}b), we consider the decomposition of $G$ 
into the subgraph $H=G^\flat \cup \set{1,2}$ and the cycle $\P$. 
Using Propositions~\ref{prop: reg bound from graph dec} and \ref{prop: additivity of reg on leaves}, we get 
$$
\textstyle \reg G \leq \reg {G^\flat} + (q-2) + \reg \P  =   \reg {G^\flat} +  \frac{\ell}{2} (q-2).\qedhere
$$
\end{proof}

\begin{definition}
Let $G$ be a bipartite graph and $\I\subset G$ an open ear of $G$. 
A bipartite ear modification of $G$ along $\I$ 
is the \emph{simple} graph obtained by either:
(i) replacing $\I$ by another open ear $\P$, 
with the same end-vertices and length of the same parity as $\ell(\I)$, or
(ii), if $\ell(\I)$ is even, by identifying the end-vertices of $\I$ in $G-\I^\circ$. 
We say that $G$ satisfies the bipartite ear modification 
hypothesis on $\I$ if, whenever $G'$ is a bipartite ear modification of $G$ along $\I$, we have
\begin{equation}\label{eq: the hypothesis}
\ts \reg G' = \reg G + \frac{|V_{G'}|-|V_G|}{2} (q-2).
\end{equation}
\end{definition}
\smallskip

Notice that, since $G$ is assumed to be bipartite, if $\ell(\I)$ is even,
then its end-vertices are not adjacent and in the bipartite ear modification described in (ii), no loop is created. 
However, in both cases, to obtain a simple graph it may be necessary to remove the multiple edges created.
\smallskip

It is easy to see that an even cycle satisfies the bipartite ear modification assumption on any of its open ears. 
Given that the regularity of a tree on $n$ vertices is $(n-2)(q-2)$, it is clear that trees do not.

\begin{theorem}\label{thm: adding a path to the endpoints of path with degree 2 inner vertices}
Let $G$ be a bipartite graph and $\I$ and $\P$ be two open ears of $G$ sharing the same end-vertices.
Let $G^\flat$ denote the graph $G-\P^\circ$, if $\ell(\P)>1$, or $G\setminus E_\P$, if
$\ell(\P)=1$. Assume that $G^\flat$ 
satisfies the bipartite ear modification hypothesis on $\I$. Then, 
$$ \textstyle \reg G = \reg G^\flat + \bigl \lfloor \frac{\ell(\P)}{2}\bigr \rfloor (q-2).$$
\end{theorem}

\begin{proof}
Note that, since $G$ is bipartite, the lengths of  $\I$ and $\P$ have the same parity.
We may assume, without loss of generality, that $\P$ is the path $(1,\dots,\ell_1+1)$,
where $\ell_1= \ell(\P)$ and $\I$ is the path $(\ell_1+1,\dots,\ell_2,1)$, where $\ell_2=\ell(\P)+\ell(\I)$, as 
illustrated in Figure~\ref{fig: adding general setup}.
\begin{figure}[ht]
\begin{center}
\begin{tikzpicture}[line cap=round,line join=round, scale=2]

\draw[fill=black] (-1.9,-.12) circle (0.2pt);
\draw[fill=black] (-1.85,-.1) circle (0.2pt);
\draw[fill=black] (-1.8,-.085) circle (0.2pt);
\draw[fill=black] (.3,-.085) circle (0.2pt);
\draw[fill=black] (0.35,-.1) circle (0.2pt);
\draw[fill=black] (0.4,-.12) circle (0.2pt);

\draw [fill=black](-1.5,0) circle (0.75pt);
\draw (-1.6,.1) node {$\scriptstyle 1$};
\draw (-1.5,0) -- (-1.6,-.2);
\draw [fill=black](0,0) circle (0.75pt);
\draw (0.25,.1) node {$\scriptstyle \ell_1+1$};
\draw (0,0) -- (.1,-.2);
\draw (0,0) -- (-.1,-.2);


\draw (-1.5,0) .. controls (-1.75,1) and (.25,1) .. (0,0);
\draw [white, fill=white](-.75,.745) circle (4pt);
\draw [fill=black](-.83,.745) circle (.2pt);
\draw [fill=black](-.75,.745) circle (.2pt);
\draw [fill=black](-.67,.745) circle (.2pt);

\draw [fill=black](-1.4,.5) circle (0.75pt);
\draw (-1.5,.55) node {$\scriptstyle 2$};
\draw [fill=black](-1,.72) circle (0.75pt);
\draw (-1.05,.82) node {$\scriptstyle 3$};

\draw [fill=black](-.5,.72) circle (0.75pt);

\draw [fill=black](-0.1,.5) circle (0.75pt);
\draw (0.1,.55) node {$\scriptstyle \ell_1$};

\draw (-.2,.75) node {$\P$};

\draw (-.35,.2) node {$\I$};


\draw (-1.75,-.07) .. controls (-1.25,.1) and (-.25,.1) .. (.25,-.07);
\draw [white, fill=white](-.75,.06) circle (4pt);
\draw [fill=black](-.83,.06) circle (.2pt);
\draw [fill=black](-.75,.06) circle (.2pt);
\draw [fill=black](-.67,.06) circle (.2pt);
\draw [fill=black](-1.25,.028) circle (0.75pt);
\draw (-1.25,-.15) node {$\scriptstyle \ell_2$};
\draw [fill=black](-1,.048) circle (0.75pt);

\draw [fill=black](-.5,.048) circle (0.75pt);
\draw [fill=black](-.25,.028) circle (0.75pt);

\end{tikzpicture}
\end{center}
\caption{}
\label{fig: adding general setup} 
\end{figure}

If the vertices $1$ and $\ell_1+1$ are neighbors in ${G^\flat}$ then, by Proposition~\ref{prop: adding a path to the endpoints of an edge}, 
the result holds.
Assume $\ell(\P)=1$ and $\ell(\I)>1$. Then the graph $G-\I^\circ$ is isomorphic to a bipartite ear modification of ${G^\flat}$ and, accordingly, 
$$
\ts \reg (G-\I^\circ) = \reg G^\flat - \lfloor \frac{\ell(\I)}{2}\bigr \rfloor (q-2).
$$
On the other hand, using again Proposition~\ref{prop: adding a path to the endpoints of an edge} we get
$$
\ts \reg G = \reg (G-\I^\circ) + \lfloor \frac{\ell(\I)}{2}\bigr \rfloor (q-2)=\reg G^\flat.
$$
Thus, from now on, we may assume that the vertices $1$ and $\ell_1+1$ are not neighbors in ${G^\flat}$ and $\ell(\I),\ell(\P)>1$.
We will split the proof into two cases according to the parity of $\ell(\I)$. 
\smallskip

We start by assuming that 
$\ell(\I)$ is odd. 
Consider the graph $G^\sharp$ obtained
by adding the edge $E=\set{1,\ell_1+1}$ to $G$. 
Then, as $G$ is a spanning subgraph of $G^\sharp$, we have $\reg G\geq \reg G^\sharp$. Denote the graph
$(G^\flat- \I^\circ)\cup E$ by $(G^\flat)'$. (See Figure~\ref{fig: the first time we use a bipartite modification}.) 
\begin{figure}[ht]
\begin{center}
\begin{tikzpicture}[line cap=round,line join=round, scale=2]

\draw[fill=black] (-1.9,-.12) circle (0.4pt);
\draw[fill=black] (-1.85,-.1) circle (0.4pt);
\draw[fill=black] (-1.8,-.085) circle (0.4pt);
\draw[fill=black] (.3,-.085) circle (0.4pt);
\draw[fill=black] (0.35,-.1) circle (0.4pt);
\draw[fill=black] (0.4,-.12) circle (0.4pt);

\draw (-1.6,.1) node {$\scriptstyle 1$};
\draw (-1.5,.55) node {$\scriptstyle 2$};

\draw [fill=black](-1.5,0) circle (.75pt);
\draw [line width = 1pt](-1.5,0) -- (-1.6,-.2);
\draw [fill=black](0,0) circle (.75pt);
\draw [line width = 1pt](0,0) -- (.1,-.2);
\draw [line width = 1pt](0,0) -- (-.1,-.2);


\draw [gray](-1.5,0) .. controls (-1.75,1) and (.25,1) .. (0,0);
\draw [white, fill=white](-.75,.745) circle (4pt);
\draw [gray,fill=gray](-.83,.745) circle (.2pt);
\draw [gray,fill=gray](-.75,.745) circle (.2pt);
\draw [gray,fill=gray](-.67,.745) circle (.2pt);

\draw [gray,fill=gray](-1.4,.5) circle (.75pt);
\draw [gray,fill=gray](-1,.72) circle (.75pt);
\draw [gray,fill=gray](-.5,.72) circle (.75pt);
\draw [gray,fill=gray](-0.1,.5) circle (.75pt);


\draw [line width = 1pt](-1.75,-.07) .. controls (-1.25,.1) and (-.25,.1) .. (.25,-.07);
\draw [line width = 2pt, white](-1.5,0) .. controls (-1.1,.07) and (-.4,.07) .. (0,0);
\draw [gray](-1.5,0) .. controls (-1.1,.07) and (-.4,.07) .. (0,0);

\draw [line width = 1pt](-1.5,0) .. controls (-1.1,-.4) and (-.4,-.4) .. (0,0);

\draw [fill=black](-1.5,0) circle (.75pt);
\draw [fill=black](0,0) circle (.75pt);

\draw [white, fill=white](-.75,.06) circle (4pt);
\draw [gray,fill=gray](-.83,.06) circle (.2pt);
\draw [gray,fill=gray](-.75,.06) circle (.2pt);
\draw [gray,fill=gray](-.67,.06) circle (.2pt);
\draw [gray,fill=gray](-1.25,.028) circle (.75pt);
\draw [gray,fill=gray](-1,.048) circle (.75pt);
\draw [gray,fill=gray](-.5,.048) circle (.75pt);
\draw [gray,fill=gray](-.25,.028) circle (.75pt);

\draw (-1.5,.55) node {$\scriptstyle 2$};
\draw (-1.05,.82) node {$\scriptstyle 3$};
\draw (0.1,.55) node {$\scriptstyle \ell_1$};
\draw (-1.6,.1) node {$\scriptstyle 1$};
\draw (0.3,.1) node {$\scriptstyle \ell_1+1$};
\draw (-1.25,.17) node {$\scriptstyle \ell_2$};

\end{tikzpicture}
\end{center}
\caption{The graph $(G^\flat)'$.}
\label{fig: the first time we use a bipartite modification} 
\end{figure}
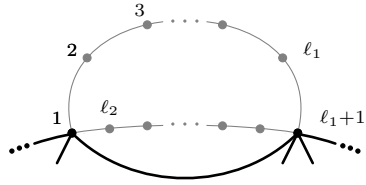

\noindent 
Since $(G^\flat)'$  
is a bipartite ear modification of $G^\flat$ along $\I$,   
$$\ts\reg (G^\flat)' = \reg G^\flat - \bigl  \lfloor \frac{\ell(\I)}{2}\bigr \rfloor (q-2).$$
On the other hand, using Proposition~\ref{prop: adding a path to the endpoints of an edge},
$$
\ts \reg G^\sharp = \reg (G^\flat)' + \big (\bigl  \lfloor \frac{\ell(\P)}{2}\bigr \rfloor + \lfloor \frac{\ell(\I)}{2}\bigr \rfloor\big )(q-2)
$$
and therefore,
$$
\ts \reg G \geq \reg G^\sharp = \reg G^\flat + \bigl  \lfloor \frac{\ell(\P)}{2}\bigr \rfloor (q-2).
$$
\smallskip

\noindent
Let us now prove the opposite inequality.
We will use induction on $\frac{\ell(\P)+\ell(\I)}{2}$. 
Consider the following monomials in $K[E_{G}]$:
\begin{equation}\label{eq:1200}
\renewcommand{\arraystretch}{1.4}
\begin{array}{l}
\t^{\delta_1} = t_{23}t_{45}\cdots t_{(\ell_1-1)\ell_1},  \\
\t^{\epsilon_1} = t_{12}t_{34}\cdots t_{\ell_1(\ell_1+1)},\\ 
\t^{\delta_2}=t_{(\ell_1+2)(\ell_1+3)}\cdots t_{(\ell_2-1)\ell_2},\\
\t^{\epsilon_2}=t_{(\ell_1+1)(\ell_1+2)}\cdots t_{\ell_21}.
\end{array}
\end{equation}
The monomial $\t^{\delta_1}$ is the monomial given by the multiplication of the variables associated to every other
edge of $\P$ starting from the second. 
The monomial $\t^{\epsilon_1}$ is the monomial given 
by the multiplication of the other edges of $\P$. The monomials $\t^{\delta_2}$ and $\t^{\epsilon_2}$ are 
described similarly with respect to $\I$. (See Figure~\ref{fig: edges in special monomials}.)
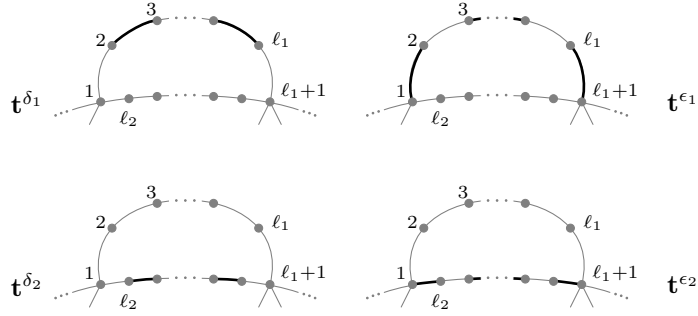
\begin{figure}[ht]
\begin{center}
\begin{tikzpicture}[line cap=round,line join=round, scale=1.5]

\draw[gray, fill=gray] (-1.9,-.12) circle (0.2pt);
\draw[gray, fill=gray] (-1.85,-.1) circle (0.2pt);
\draw[gray, fill=gray] (-1.8,-.085) circle (0.2pt);
\draw[gray, fill=gray] (.3,-.085) circle (0.2pt);
\draw[gray, fill=gray] (0.35,-.1) circle (0.2pt);
\draw[gray, fill=gray] (0.4,-.12) circle (0.2pt);

\draw(-2.15,0) node {$\t^{\delta_1}$};

\draw [gray](-1.5,0) -- (-1.6,-.2);
\draw [gray](0,0) -- (.1,-.2);
\draw [gray](0,0) -- (-.1,-.2);
\draw [gray](-1.5,0) .. controls (-1.75,1) and (.25,1) .. (0,0);
\draw [white, fill=white](-.75,.745) circle (4pt);
\draw [gray, fill=gray](-.83,.745) circle (.2pt);
\draw [gray, fill=gray](-.75,.745) circle (.2pt);
\draw [gray, fill=gray](-.67,.745) circle (.2pt);
\draw [gray] (-1.75,-.07) .. controls (-1.25,.1) and (-.25,.1) .. (.25,-.07);
\draw [white, fill=white](-.75,.06) circle (4pt);
\draw [gray, fill=gray](-.83,.06) circle (.2pt);
\draw [gray, fill=gray](-.75,.06) circle (.2pt);
\draw [gray, fill=gray](-.67,.06) circle (.2pt);


\draw [line width = 1pt] (-1.4,.5)..controls (-1.35,.57) and (-1.1,.73).. (-1,.72);
\draw [line width = 1pt] (-.5,.72) ..controls (-.4,.72) and (-.2,.63) .. (-0.1,.5);

\draw [gray,fill=gray](-1.4,.5) circle (1pt);
\draw (-1.5,.55) node {$\scriptstyle 2$};
\draw [gray, fill=gray](-1,.72) circle (1pt);
\draw (-1.05,.82) node {$\scriptstyle 3$};
\draw [gray, fill=gray](-.5,.72) circle (1pt);
\draw [gray, fill=gray](-0.1,.5) circle (1pt);
\draw (0.1,.55) node {$\scriptstyle \ell_1$};
\draw [gray, fill=gray](-1.5,0) circle (1pt);
\draw (-1.6,.1) node {$\scriptstyle 1$};
\draw [gray, fill=gray](0,0) circle (1pt);
\draw (0.3,.1) node {$\scriptstyle \ell_1+1$};
\draw [gray, fill=gray](-1.25,.028) circle (1pt);
\draw (-1.25,-.15) node {$\scriptstyle \ell_2$};
\draw [gray, fill=gray](-1,.048) circle (1pt);
\draw [gray, fill=gray](-.5,.048) circle (1pt);
\draw [gray, fill=gray](-.25,.028) circle (1pt);

\draw[white] (.5,-.5) circle (2.5pt);
\end{tikzpicture}
\begin{tikzpicture}[line cap=round,line join=round, scale=1.5]

\draw(.9,0) node {$\t^{\epsilon_1}$};

\draw[gray, fill=gray] (-1.9,-.12) circle (0.2pt);
\draw[gray, fill=gray] (-1.85,-.1) circle (0.2pt);
\draw[gray, fill=gray] (-1.8,-.085) circle (0.2pt);
\draw[gray, fill=gray] (.3,-.085) circle (0.2pt);
\draw[gray, fill=gray] (0.35,-.1) circle (0.2pt);
\draw[gray, fill=gray] (0.4,-.12) circle (0.2pt);

\draw [gray](-1.5,0) -- (-1.6,-.2);
\draw [gray](0,0) -- (.1,-.2);
\draw [gray](0,0) -- (-.1,-.2);
\draw [gray](-1.5,0) .. controls (-1.75,1) and (.25,1) .. (0,0);
\draw [white, fill=white](-.75,.745) circle (4pt);
\draw [gray, fill=gray](-.83,.745) circle (.2pt);
\draw [gray, fill=gray](-.75,.745) circle (.2pt);
\draw [gray, fill=gray](-.67,.745) circle (.2pt);
\draw [gray] (-1.75,-.07) .. controls (-1.25,.1) and (-.25,.1) .. (.25,-.07);
\draw [white, fill=white](-.75,.06) circle (4pt);
\draw [gray, fill=gray](-.83,.06) circle (.2pt);
\draw [gray, fill=gray](-.75,.06) circle (.2pt);
\draw [gray, fill=gray](-.67,.06) circle (.2pt);


\draw [line width = 1pt] (-1.5,0) ..controls (-1.55,.15) and (-1.5,.35) ..(-1.4,.5);
\draw [line width = 1pt] (-0.1,.5) .. controls (0,.4) and (.05,.15) .. (0,0);
\draw [line width = 1pt] (-1,.72) .. controls (-.8,.76) and (-.7,.76) .. (-.5,.72);


\draw [white, fill=white](-.75,.745) circle (4pt);
\draw [gray, fill=gray](-.83,.745) circle (.2pt);
\draw [gray, fill=gray](-.75,.745) circle (.2pt);
\draw [gray, fill=gray](-.67,.745) circle (.2pt);

\draw [gray,fill=gray](-1.4,.5) circle (1pt);
\draw (-1.5,.55) node {$\scriptstyle 2$};
\draw [gray, fill=gray](-1,.72) circle (1pt);
\draw (-1.05,.82) node {$\scriptstyle 3$};
\draw [gray, fill=gray](-.5,.72) circle (1pt);
\draw [gray, fill=gray](-0.1,.5) circle (1pt);
\draw (0.1,.55) node {$\scriptstyle \ell_1$};
\draw [gray, fill=gray](-1.5,0) circle (1pt);
\draw (-1.6,.1) node {$\scriptstyle 1$};
\draw [gray, fill=gray](0,0) circle (1pt);
\draw (0.3,.1) node {$\scriptstyle \ell_1+1$};
\draw [gray, fill=gray](-1.25,.028) circle (1pt);
\draw (-1.25,-.15) node {$\scriptstyle \ell_2$};
\draw [gray, fill=gray](-1,.048) circle (1pt);
\draw [gray, fill=gray](-.5,.048) circle (1pt);
\draw [gray, fill=gray](-.25,.028) circle (1pt);

\draw[white] (-2,-.5) circle (2.5pt);

\end{tikzpicture}

\begin{tikzpicture}[line cap=round,line join=round, scale=1.5]

\draw(-2.15,0) node {$\t^{\delta_2}$};

\draw[gray, fill=gray] (-1.9,-.12) circle (0.2pt);
\draw[gray, fill=gray] (-1.85,-.1) circle (0.2pt);
\draw[gray, fill=gray] (-1.8,-.085) circle (0.2pt);
\draw[gray, fill=gray] (.3,-.085) circle (0.2pt);
\draw[gray, fill=gray] (0.35,-.1) circle (0.2pt);
\draw[gray, fill=gray] (0.4,-.12) circle (0.2pt);

\draw [gray](-1.5,0) -- (-1.6,-.2);
\draw [gray](0,0) -- (.1,-.2);
\draw [gray](0,0) -- (-.1,-.2);
\draw [gray](-1.5,0) .. controls (-1.75,1) and (.25,1) .. (0,0);
\draw [white, fill=white](-.75,.745) circle (4pt);
\draw [gray, fill=gray](-.83,.745) circle (.2pt);
\draw [gray, fill=gray](-.75,.745) circle (.2pt);
\draw [gray, fill=gray](-.67,.745) circle (.2pt);
\draw [gray] (-1.75,-.07) .. controls (-1.25,.1) and (-.25,.1) .. (.25,-.07);
\draw [white, fill=white](-.75,.06) circle (4pt);
\draw [gray, fill=gray](-.83,.06) circle (.2pt);
\draw [gray, fill=gray](-.75,.06) circle (.2pt);
\draw [gray, fill=gray](-.67,.06) circle (.2pt);


\draw [line width = 1pt] (-1.25,.028)..controls (-1.15,.04) and (-1.1,.05).. (-1,.048);
\draw [line width = 1pt] (-.5,.048) ..controls (-.4,.05) and (-.25,.04) .. (-0.25,.028);

\draw [gray,fill=gray](-1.4,.5) circle (1pt);
\draw (-1.5,.55) node {$\scriptstyle 2$};
\draw [gray, fill=gray](-1,.72) circle (1pt);
\draw (-1.05,.82) node {$\scriptstyle 3$};
\draw [gray, fill=gray](-.5,.72) circle (1pt);
\draw [gray, fill=gray](-0.1,.5) circle (1pt);
\draw (0.1,.55) node {$\scriptstyle \ell_1$};
\draw [gray, fill=gray](-1.5,0) circle (1pt);
\draw (-1.6,.1) node {$\scriptstyle 1$};
\draw [gray, fill=gray](0,0) circle (1pt);
\draw (0.3,.1) node {$\scriptstyle \ell_1+1$};
\draw [gray, fill=gray](-1.25,.028) circle (1pt);
\draw (-1.25,-.15) node {$\scriptstyle \ell_2$};
\draw [gray, fill=gray](-1,.048) circle (1pt);
\draw [gray, fill=gray](-.5,.048) circle (1pt);
\draw [gray, fill=gray](-.25,.028) circle (1pt);

\draw[white] (.5,-.5) circle (2.5pt);
\end{tikzpicture}
\begin{tikzpicture}[line cap=round,line join=round, scale=1.5]

\draw[gray, fill=gray] (-1.9,-.12) circle (0.2pt);
\draw[gray, fill=gray] (-1.85,-.1) circle (0.2pt);
\draw[gray, fill=gray] (-1.8,-.085) circle (0.2pt);
\draw[gray, fill=gray] (.3,-.085) circle (0.2pt);
\draw[gray, fill=gray] (0.35,-.1) circle (0.2pt);
\draw[gray, fill=gray] (0.4,-.12) circle (0.2pt);

\draw(.9,0) node {$\t^{\epsilon_2}$};

\draw [gray](-1.5,0) -- (-1.6,-.2);
\draw [gray](0,0) -- (.1,-.2);
\draw [gray](0,0) -- (-.1,-.2);
\draw [gray](-1.5,0) .. controls (-1.75,1) and (.25,1) .. (0,0);
\draw [white, fill=white](-.75,.745) circle (4pt);
\draw [gray, fill=gray](-.83,.745) circle (.2pt);
\draw [gray, fill=gray](-.75,.745) circle (.2pt);
\draw [gray, fill=gray](-.67,.745) circle (.2pt);
\draw [gray] (-1.75,-.07) .. controls (-1.25,.1) and (-.25,.1) .. (.25,-.07);
\draw [white, fill=white](-.75,.06) circle (4pt);
\draw [gray, fill=gray](-.83,.06) circle (.2pt);
\draw [gray, fill=gray](-.75,.06) circle (.2pt);
\draw [gray, fill=gray](-.67,.06) circle (.2pt);


\draw [line width = 1pt] (-1.5,0) ..controls (-1.4,.015) and (-1.35,.02) ..(-1.25,.028);
\draw [line width = 1pt] (-1,.048) .. controls (-.8,.06) and (-.7,.06) .. (-.5,.048);
\draw [line width = 1pt] (0,0) ..controls (-0.1,.015) and (-0.35,.05) ..(-0.25,.028);

\draw [white, fill=white](-.75,.06) circle (4pt);
\draw [gray, fill=gray](-.83,.06) circle (.2pt);
\draw [gray, fill=gray](-.75,.06) circle (.2pt);
\draw [gray, fill=gray](-.67,.06) circle (.2pt);

\draw [gray,fill=gray](-1.4,.5) circle (1pt);
\draw (-1.5,.55) node {$\scriptstyle 2$};
\draw [gray, fill=gray](-1,.72) circle (1pt);
\draw (-1.05,.82) node {$\scriptstyle 3$};
\draw [gray, fill=gray](-.5,.72) circle (1pt);
\draw [gray, fill=gray](-0.1,.5) circle (1pt);
\draw (0.1,.55) node {$\scriptstyle \ell_1$};
\draw [gray, fill=gray](-1.5,0) circle (1pt);
\draw (-1.6,.1) node {$\scriptstyle 1$};
\draw [gray, fill=gray](0,0) circle (1pt);
\draw (0.3,.1) node {$\scriptstyle \ell_1+1$};
\draw [gray, fill=gray](-1.25,.028) circle (1pt);
\draw (-1.25,-.15) node {$\scriptstyle \ell_2$};
\draw [gray, fill=gray](-1,.048) circle (1pt);
\draw [gray, fill=gray](-.5,.048) circle (1pt);
\draw [gray, fill=gray](-.25,.028) circle (1pt);

\draw[white] (-2,-.5) circle (2.5pt);
\end{tikzpicture}

\end{center}
\vspace{-.5cm}
\caption{Edges in the support of $\t^{\delta_1}$, $\t^{\epsilon_1}$, $\t^{\delta_2}$ and $\t^{\epsilon_2}$.}
\label{fig: edges in special monomials} 
\end{figure}

\noindent 
Notice that 
$$\renewcommand{\arraystretch}{1.75}
\begin{array}{cc}
\deg(\t^{\delta_1}) = \bigl \lfloor \frac{\ell(\P)}{2} \bigr\rfloor,  &
\deg(\t^{\epsilon_1}) = \bigl\lfloor \frac{\ell(\P)}{2} \bigr \rfloor +1,  \\
\deg(\t^{\delta_2}) = \bigl\lfloor \frac{\ell(\I)}{2} \bigr\rfloor,  &
\deg(\t^{\epsilon_2}) = \bigl\lfloor \frac{\ell(\I)}{2} \bigr \rfloor +1. \\
\end{array}
$$
Also, since $\t^{\epsilon_1}\t^{\delta_2} - \t^{\delta_1}\t^{\epsilon_2}$
is a generator of the toric ideal of the even cycle $\P\cup \I$,
\begin{equation}\label{eq:643}
\t^{\epsilon_1}\t^{\delta_2} - \t^{\delta_1}\t^{\epsilon_2} \in I_q(\P\cup \I)\subset I_q(G).
\end{equation}
By Proposition~\ref{prop: computing reg by reducing to Artinian quotient}, to show that 
\begin{equation}\label{eq:L888}
\ts \reg G \leq  \reg G^\flat + \bigl  \lfloor \frac{\ell(\P)}{2}\bigr \rfloor (q-2)
\end{equation}
it suffices to prove that for any monomial $\t^\nu\in K[E_{G}]$ of degree 
$$
\ts
  \reg G^\flat +\bigl \lfloor \frac{\ell(\P)}{2} \bigr\rfloor (q-2) +  
 \bigl \lfloor \frac{\ell(\P)}{2}\bigr \rfloor +  
 \bigl \lfloor \frac{\ell(\I)}{2}\bigr \rfloor
$$
there exists $\t^\mu \in K[E_{G}]$, of the same degree as $\t^\nu$, divisible by $\t^{\delta_1} \t^{\delta_2}$, 
such that \mbox{$\t^\nu - \t^\mu$} belongs to $I_q(G)$.
Set 
\mbox{$\t^\nu = \t^{\alpha}\t^{\beta}\t^{\gamma}$} for some $\alpha,\beta,\gamma \in \NN^{E_{G}}$ satisfying  
\mbox{$\t^\alpha \in K[E_{\P}]$,} $\t^\beta \in K[E_{\I}]$, \mbox{$\t^\gamma \in K[E_{G^\flat-\I^\circ}]$}. Then,
\begin{equation}\label{eq:L805}
\textstyle \deg(\t^\alpha) + \deg (\t^\beta) + \deg(\t^\gamma) = \reg G^\flat + \bigl \lfloor \frac{\ell(\P)}{2}\bigr \rfloor (q-1)  
+  \bigl \lfloor \frac{\ell(\I)}{2}\bigr \rfloor.
\end{equation}

\noindent
Suppose that 
\begin{equation}\label{eq:617}
\textstyle \deg(\t^\alpha) \geq \bigl \lfloor \frac{\ell(\P)}{2}\bigr \rfloor (q-1) \quad 
\text{and}\quad  \deg(\t^\beta) \geq \bigl \lfloor \frac{\ell(\I)}{2}\bigr \rfloor (q-1).
\end{equation}
Then, $\t^\alpha\t^\beta$, which is supported on the cycle $\P\cup \I$, is such that 
$$
\ts \deg (\t^\alpha\t^\beta) \geq \reg (\P\cup \I) + \lfloor \frac{\ell(\P)}{2}\bigr \rfloor + \lfloor \frac{\ell(\I)}{2}\bigr \rfloor 
$$
and then, by Proposition~\ref{prop: computing reg by reducing to Artinian quotient}, applied to $\P\cup\I$,
there exists $\t^\mu\in K[E_{\P\cup \I}]$, divisible 
by $\t^{\delta_1}\t^{\delta_2}$ such that $\t^\alpha \t^\beta \in I_q(\P\cup \I)\subset I_q(G)$. We deduce that
$$
\textstyle \t^\alpha\t^\beta\t^\gamma - \t^\mu\t^\gamma \in I_q(G),
$$ 
as desired. Assume now that (\ref{eq:617}) does not hold. Now, directly from \eqref{eq:L805}, 
$$
\ts \deg(\t^\alpha)< \bigl\lfloor \frac{\ell(\P)}{2}\bigr \rfloor (q-1) \iff \deg (\t^\beta \t^\gamma) \geq \reg G^\flat + \bigl \lfloor \frac{\ell(\I)}{2}\bigr \rfloor +1.
$$
On the other hand, since $G- \I^\circ$ is a bipartite ear modification of $G^\flat$ along $\I$, and therefore by our 
assumptions, 
$$
\ts \reg (G-\I^\circ) = \reg G^\flat + \bigl\lfloor \frac{\ell(\P)}{2}\bigr \rfloor (q-2) - \bigl\lfloor \frac{\ell(\I)}{2}\bigr \rfloor (q-2),
$$
we get from \eqref{eq:L805}:
$$
\ts \deg(\t^\beta)< \bigl\lfloor \frac{\ell(\I)}{2}\bigr \rfloor (q-1) \iff \deg (\t^\alpha \t^\gamma) \geq \reg (G-\I^\circ) + \bigl \lfloor \frac{\ell(\P)}{2}\bigr \rfloor +1.
$$
Hence, by symmetry, we may assume that  
\begin{equation}\label{eq:634}
\ts \deg(\t^\alpha)< \bigl\lfloor \frac{\ell(\P)}{2}\bigr \rfloor (q-1) \iff \deg (\t^\beta \t^\gamma) \geq \reg G^\flat + \bigl \lfloor \frac{\ell(\I)}{2}\bigr \rfloor +1.
\end{equation}
Then, by Proposition~\ref{prop: computing reg by reducing to Artinian quotient}, 
there exists $\t^\mu \in K[E_{G^\flat}]$, of degree equal to $\deg(\t^\beta \t^\gamma)$, divisible by $\t^{\delta_2}$, 
such that 
$$
\t^\beta \t^\gamma -\t^\mu \in I_q(G^\flat)\subset I_q(G),
$$  
which implies that  
$$\t^\alpha \t^\beta \t^\gamma - \t^\alpha \t^\mu \in I_q(G).$$ 
If $\t^{\delta_1}$ divides $\t^\alpha$ we have finished. 
Assume $\t^{\delta_1}$ does not divide $\t^\alpha$.
If $\t^{\epsilon_1}$ divides $\t^\alpha$, then, since 
$$\t^\alpha\t^\mu  = (\t^\alpha\t^{-\epsilon_1})(\t^{\epsilon_1}\t^{\delta_2})(\t^{-\delta_2}\t^\mu )$$ 
and $\t^{\epsilon_1}\t^{\delta_2}-\t^{\delta_1}\t^{\epsilon_2}\in I_q(G)$ we get:
\begin{equation}\label{eq: L800}
\ts \t^\alpha\t^\beta \t^\gamma - (\t^\alpha\t^{-\epsilon_1}\t^{\delta_1} )(\t^{\epsilon_2} \t^{-\delta_2}\t^\mu ) \in I_q(G),
\end{equation}
where, 
$$
\ts \deg( \t^{\epsilon_2} \t^{-\delta_2}\t^\mu) = \deg(\t^\beta \t^\gamma) +1 \geq \reg G^\flat + \bigl \lfloor \frac{\ell(\I)}{2}\bigr \rfloor +2.
$$
Hence, by Proposition~\ref{prop: computing reg by reducing to Artinian quotient}, there exists 
$\t^\rho \in K[E_{G^\flat}]$ divisible by $\t^{\delta_2}$ such that 
\begin{equation}\label{eq: L809}
 \t^{\epsilon_2} \t^{-\delta_2}\t^\mu -\t^\rho \in I_q(G^\flat)\subset I_q(G).
\end{equation}
From \eqref{eq: L800} and \eqref{eq: L809}, we deduce that 
$$
\t^\alpha \t^\beta \t^\gamma - (\t^\alpha\t^{-\epsilon_1}\t^{\delta_1} )\t^\rho \in I_q(G)
$$ 
where 
$(\t^\alpha\t^{-\epsilon_1}\t^{\delta_1} )\t^\rho$ is divisible by $\t^{\delta_1}\t^{\delta_2}$, as required.
\smallskip

\noindent

\smallskip

\noindent
We may assume from now on that $\t^\alpha$ is divisible by neither $\t^{\epsilon_1}$ nor $\t^{\delta_1}$. 
Since showing that there exists a monomial $\t^\mu \in K[E_{G}]$ of degree equal to the degree of 
$\t^\nu = \t^\alpha \t^\beta \t^\gamma$, divisible by $\t^{\delta_1}\t^{\delta_2}$ such that $\t^\nu - \t^\mu \in I_q(G)$ 
is, by \cite[Lemma~2.1]{MaNeVPVi}, equivalent to showing that the same holds for the monomial
obtained by permuting the variables of the support of $\t^{\delta_1}$ and permuting the variables of the support of
$\t^{\epsilon_1}$, we may assume that
neither $t_{12}$ nor $t_{23}$ divides $\t^\nu$. 
\smallskip

\noindent
Consider the graph $H$ obtained 
from $G$ by removing the edges $\set{1,2}$ and $\set{2,3}$,
and identifying the vertices $1$ and $3$, as
illustrated in Figure~\ref{fig:Contracting the a 2-path on I}.  
\begin{figure}[ht]
\begin{center}
\begin{tikzpicture}[line cap=round,line join=round, scale=1.75]

\draw[black, fill=black] (-1.9,-.12) circle (0.2pt);
\draw[black, fill=black] (-1.85,-.1) circle (0.2pt);
\draw[black, fill=black] (-1.8,-.085) circle (0.2pt);
\draw[black, fill=black] (.3,-.085) circle (0.2pt);
\draw[black, fill=black] (0.35,-.1) circle (0.2pt);
\draw[black, fill=black] (0.4,-.12) circle (0.2pt);

\draw [fill=black](-1.5,0) circle (1pt);
\draw (-1.6,.1) node {$\scriptstyle 1$};
\draw (-1.5,0) -- (-1.6,-.2);
\draw [fill=black](0,0) circle (1pt);
\draw (0.3,.1) node {$\scriptstyle \ell_1+1$};
\draw (0,0) -- (.1,-.2);
\draw (0,0) -- (-.1,-.2);


\draw (-1.5,0) .. controls (-1.75,1) and (.25,1) .. (0,0);
\draw [white, fill=white](-.75,.745) circle (4pt);
\draw [fill=black](-.83,.745) circle (.2pt);
\draw [fill=black](-.75,.745) circle (.2pt);
\draw [fill=black](-.67,.745) circle (.2pt);

\draw [fill=black](-1.4,.5) circle (1pt);
\draw (-1.5,.55) node {$\scriptstyle 2$};
\draw [fill=black](-1,.72) circle (1pt);
\draw (-1.05,.82) node {$\scriptstyle 3$};

\draw [fill=black](-.5,.72) circle (1pt);

\draw [fill=black](-0.1,.5) circle (1pt);
\draw (0.1,.55) node {$\scriptstyle \ell_1$};


\draw (-1.75,-.07) .. controls (-1.25,.1) and (-.25,.1) .. (.25,-.07);
\draw [white, fill=white](-.75,.06) circle (4pt);
\draw [fill=black](-.83,.06) circle (.2pt);
\draw [fill=black](-.75,.06) circle (.2pt);
\draw [fill=black](-.67,.06) circle (.2pt);
\draw [fill=black](-1.25,.028) circle (1pt);
\draw (-1.25,-.15) node {$\scriptstyle \ell_2$};
\draw [fill=black](-1,.048) circle (1pt);

\draw [fill=black](-.5,.048) circle (1pt);
\draw [fill=black](-.25,.028) circle (1pt);

\draw(-2.25,0) node {$G$};

\draw [white] (1,-.4) circle (.5pt);
\draw [white] (-2.5,-.4) circle (.5pt);
\draw [white] (1,.85) circle (.5pt);
\draw [white] (-2.5,.85) circle (.5pt);

\end{tikzpicture}

\begin{tikzpicture}[line cap=round,line join=round, scale=1.5]
\BDarrow{0}
\end{tikzpicture}

\begin{tikzpicture}[line cap=round,line join=round, scale=1.75]

\draw[black, fill=black] (-1.9+4,-.12) circle (0.2pt);
\draw[black, fill=black] (-1.85+4,-.1) circle (0.2pt);
\draw[black, fill=black] (-1.8+4,-.085) circle (0.2pt);
\draw[black, fill=black] (.3+4,-.085) circle (0.2pt);
\draw[black, fill=black] (0.35+4,-.1) circle (0.2pt);
\draw[black, fill=black] (0.4+4,-.12) circle (0.2pt);

\draw [fill=black](2.5,0) circle (1pt);
\draw[gray] (2.25,.1) node {$\ss 3=$};
\draw (2.4,.1) node {$\ss 1$};
\draw (2.5,0) -- (2.4,-.2);
\draw [fill=black](4,0) circle (1pt);
\draw (4.3,.1) node {$\scriptstyle \ell_1+1$};
\draw (4,0) -- (4.1,-.2);
\draw (4,0) -- (3.9,-.2);


\draw (2.5,0) .. controls (2.25,.5) and (4.25,1) .. (4,0);
\draw [white, fill=white](3.25,.645) circle (5pt);
\draw [fill=black](3.17,.55) circle (.2pt);
\draw [fill=black](3.25,.56) circle (.2pt);
\draw [fill=black](3.33,.57) circle (.2pt);

\draw [fill=black](2.95,.5) circle (1pt);
\draw (2.87,.63) node {$\scriptstyle 4$};
 
\draw [fill=black](3.55,.57) circle (1pt);

 \draw [fill=black](3.9,.45) circle (1pt);
 \draw (4.1,.55) node {$\scriptstyle \ell_1$};

%
\draw (2.25,-.07) .. controls (2.75,.1) and (3.75,.1) .. (4.25,-.07);
\draw [white, fill=white](3.25,.06) circle (4pt);
\draw [fill=black](3.17,.06) circle (.2pt);
\draw [fill=black](3.25,.06) circle (.2pt);
\draw [fill=black](3.33,.06) circle (.2pt);
\draw [fill=black](2.75,.028) circle (1pt);
\draw (2.75,-.15) node {$\scriptstyle \ell_2$};
\draw [fill=black](3,.048) circle (1pt);
\draw [fill=black](3.5,.048) circle (1pt);
\draw [fill=black](3.75,.028) circle (1pt);

\draw(1.75,0) node {$H$};

\draw [white] (5,-.3) circle (.5pt);
\draw [white] (1.5,-.3) circle (.5pt);
\draw [white] (5,.85) circle (.5pt);
\draw [white] (1.5,.85) circle (.5pt);

\end{tikzpicture}
\end{center}
\caption{}
\label{fig:Contracting the a 2-path on I} 
\end{figure}
Denote the ear obtained from $\P$ after this operation by $\Q$.  
By induction: 
$$
\ts \reg H = \reg (H - \Q^\circ) + \bigl \lfloor \frac{\ell(\Q)}{2} \bigr \rfloor (q-2) = \reg G^\flat + \bigl \lfloor \frac{\ell(\P)}{2} \bigr \rfloor (q-2)-(q-2).
$$
Let $d=\nu\ss\set{3,4}$ and let $\t^{\overline{\nu}}\in K[E_H]$ be given by:
$$
\t^{\overline{\nu}}=\t^\nu t_{34}^{-d}t_{14}^d.
$$
Then, since 
$$
\ts \deg(\t^{\overline{\nu}})=\deg(\t^\nu) = \reg H + \bigl \lfloor \frac{\ell(\P)}{2} \bigr \rfloor + \bigl \lfloor \frac{\ell(\I)}{2} \bigr \rfloor  + (q-2)
$$
and $\t^{\delta_1}t_{23}^{-1}\t^{\delta_2}t_{1\ell_2}^{q-1}\in K[E_H]$ is such that
$$
\ts \deg (\t^{\delta_1}t_{23}^{-1}\t^{\delta_2}t_{1\ell_2}^{q-1}) = \bigl \lfloor \frac{\ell(\P)}{2} \bigr \rfloor + \bigl \lfloor \frac{\ell(\I)}{2} \bigr \rfloor  + (q-2),
$$
by Proposition~\ref{prop: computing reg by reducing to Artinian quotient} 
applied to the graph $H$, there exists $\t^{\overline{\mu}} \in K[E_{H}]$ such that 
$\t^{\delta_1}t_{23}^{-1}\t^{\delta_2}t_{1\ell_2}^{q-1}$ divides $\t^{\overline{\mu}}$ 
and  $\t^{\overline{\nu}} - \t^{\overline{\mu}} \in I_q(H)$.
Let $c=\overline{\mu}\ss\set{1,4}$ and let $\t^\mu\in K[E_G]$ be given by:
$$
\t^\mu = \t^{\overline{\mu}}t_{14}^{-c}t_{34}^c.
$$
By Proposition~\ref{prop: congruences},
the binomial $\t^{\overline{\nu}}-\t^{\overline{\mu}}$ satisfies a set of congruences modulo $q-1$, one for each vertex of $H$. 
In particular, at $1\in V_{H}$, we have:
\begin{equation}\label{eq: L1427}
d + \hspace{-.2cm}\sum_{k\in N_G(1)} \hspace{-.2cm} \nu\sset{1,k} \ds \hspace{.1cm} =  \hspace{-.2cm} \sum_{k\in N_{H}(1)} \hspace{-.2cm} \overline{\nu}\sset{1,k}\ds \hspace{.2cm}  
\equiv \sum_{k\in N_{H}(1)} \hspace{-.2cm} \overline{\mu}\sset{1,k}\ds \hspace{.1cm} = \hspace{.1cm}
c+\hspace{-.2cm} \sum_{k\in N_G(1)} \hspace{-.2cm}\mu\sset{1,k}.
\end{equation}
Let $a\in \set{1,\dots,q-1}$ be such that $a\equiv d-c$ and let $b=(q-1)-a$.
Then, as 
$\t^{\delta_1}t_{23}^{-1}\t^{\delta_2}t_{1\ell_2}^{q-1}$ divides $\t^{\overline{\mu}}$ and hence it divides $\t^\mu$, the binomial 
\begin{equation}\label{eq: L1696}
\t^\nu - \t^\mu t_{1\ell_2}^{-(q-1)} t_{12}^{b} t_{23}^{a}
\end{equation}
is a homogeneous binomial of the ring $K[E_{G}]$. Moreover, since 
$a\geq 1$, we deduce that $\t^{\delta_1}\t^{\delta_2}$ divides the monomial on the right side of \eqref{eq: L1696}.
\smallskip

\noindent
Let us prove that the binomial \eqref{eq: L1696}  belongs to $I_q(G)$. 
It suffices to check that the corresponding congruences at vertices $1$, $2$ and $3$ of $G$ are satisfied, since at any other
vertex the corresponding congruence is identical to the one in $H$. At the vertices
$2$ and $3$, we get, respectively,
$$
0\equiv a+b\quad\text{and}\quad d \equiv  c +a  \iff d-c \equiv a
$$
which hold, by the definitions of $b$ and $a$.
At the vertex $1$, we get:
$$
 \sum_{k\in N_G(1)}  \nu\sset{1,k} \ds  \equiv b -(q-1)+ \hspace{-.2cm} \sum_{k\in N_G(1)} \mu\sset{1,k} \ds   = c-d +\hspace{-.2cm} \sum_{k\in N_G(1)} \mu\sset{1,k},
$$
which holds by \eqref{eq: L1427}. Hence \eqref{eq: L1696} belongs to $I_q(G)$.
This concludes the proof, by induction, of the inequality \eqref{eq:L888} in the case of 
$\ell(\I)$ odd.
\smallskip

Let us now consider the case of $\ell(\I)$ and $\ell(\P)$ even. We start by proving that 
\begin{equation}\label{eq: L1222}
\ts \reg G \geq \reg G^\flat + \frac{\ell(\P)}{2}(q-2).
\end{equation}
Consider the simple graph $H$ obtained from $G$ by identifying the vertices $1$ and $\ell_1+1$ and denote by 
$H'$ be 
the subgraph of $H$ obtained from $G^\flat-\I^\circ$ under the same identification. (See Figure~\ref{fig: Identifying the end-vertices}.)
\begin{figure}[ht]
\begin{center}
\begin{tikzpicture}[line cap=round,line join=round, scale=1.75]

\draw[black, fill=black] (-1.9+.75,-.12-.8) circle (0.2pt);
\draw[black, fill=black] (-1.85+.75,-.1-.8) circle (0.2pt);
\draw[black, fill=black] (-1.8+.75,-.085-.8) circle (0.2pt);
\draw[black, fill=black] (.3-1.5+.75,-.085-.8) circle (0.2pt);
\draw[black, fill=black] (0.35-1.5+.75,-.1-.8) circle (0.2pt);
\draw[black, fill=black] (0.4-1.5+.75,-.12-.8) circle (0.2pt);

\draw [fill=black](-.75,-.8) circle (1pt);
\draw (-0.75,-.8) -- (-1.6+.65,-1);
\draw (-.75,-.8) -- (.1-.75,-.2-.8);
\draw (-.75,-.8) -- (-.1-.75,-.2-.8);

\draw (-.10,-.78) node {$\scriptstyle 1=\ell_1 +1 $};


\draw (-.75,-.8) .. controls (-2.75-.6,1.265) and (1.25+.6,1.265) .. (-.75,-.8);
\draw [white, fill=white](-.75,.745) circle (4pt);
\draw [fill=black](-.83,.745) circle (.2pt);
\draw [fill=black](-.75,.745) circle (.2pt);
\draw [fill=black](-.67,.745) circle (.2pt);

\draw [fill=black](-1.4,.5) circle (1pt);
\draw (-1.5,.55) node {$\scriptstyle 2$};
\draw [fill=black](-1,.72) circle (1pt);
\draw (-1.05,.82) node {$\scriptstyle 3$};
\draw [fill=black](-.5,.72) circle (1pt);
\draw [fill=black](-0.1,.5) circle (1pt);
\draw (0.1,.55) node {$\scriptstyle \ell_1$};

\draw (-.052,-.03) [white, fill=white] circle (1.5pt);
\draw (-1.43,-.03) [white, fill=white] circle (1.5pt);


\draw (-1.75+1.2,-.87) .. controls (-1.25-3,.37) and (-.25+3,.37) .. (.25-1.2,-.87);
\draw [white, fill=white](-.75,.06) circle (4pt);
\draw [fill=black](-.83,.06) circle (.2pt);
\draw [fill=black](-.75,.06) circle (.2pt);
\draw [fill=black](-.67,.06) circle (.2pt);
\draw [fill=black](-1.25,.02) circle (1pt);
\draw (-1.25,.15) node {$\scriptstyle \ell_2$};
\draw [fill=black](-1,.048) circle (1pt);

\draw [fill=black](-.5,.048) circle (1pt);
\draw [fill=black](-.25,.02) circle (1pt);

\draw [white] (.5,-1) circle (.5pt);
\draw [white] (-2,-1) circle (.5pt);
\draw [white] (.5,.95) circle (.5pt);
\draw [white] (-2,.95) circle (.5pt);

\end{tikzpicture}
\end{center}
\caption{}
\label{fig: Identifying the end-vertices} 
\end{figure}

\noindent
Since $H'$ is a bipartite ear modification of $G^\flat$,
$$
\ts \reg H' = \reg G^\flat - \frac{\ell(I)}{2}(q-2).
$$
On the other hand, using Propositions~\ref{prop: vertex identification} and \ref{prop: adding a path to the endpoints of an edge}, or, in
the case of $\ell(\P)=2$ or $\ell(\I)=2$, Proposition~\ref{prop: additivity of reg on leaves},
$$
\ts \reg G \geq \reg H = \reg H' + \frac{\ell(\P)}{2}(q-2) + \frac{\ell(\I)}{2}(q-2)= \reg G^\flat + \frac{\ell(\P)}{2}(q-2),
$$
which proves \eqref{eq: L1222}. To prove that 
\begin{equation}\label{eq: L1824}
\ts \reg G \leq \reg G^\flat + \frac{\ell(\P)}{2}(q-2),
\end{equation}
by Proposition~\ref{prop: computing reg by reducing to Artinian quotient}, it suffices to show that 
for every $\t^\nu\in K[E_{G}]$ with 
\begin{equation}\label{eq: L296}
\ts\deg(\t^\nu) = \reg G^\flat + \frac{\ell(\P)}{2}(q-2)+1,
\end{equation}
there exists $\t^\mu$, divisible by $t_{12}$, such that $\t^\nu - \t^\mu\in I_q(G)$.
Consider the following graphs:
$$
(G^\flat)^*= G^\flat \cup \set{1,2}, \quad \C = \P\cup \I\quad \text{and} \quad G-\I^\circ. 
$$
(See Figure~\ref{fig: the 3 graphs in the decomposition}.)
\begin{figure}[ht]
\begin{center}
\begin{tikzpicture}[line cap=round,line join=round, scale=1.8]

\draw[black, fill=black] (-1.9,-.12) circle (0.4pt);
\draw[black, fill=black] (-1.85,-.1) circle (0.4pt);
\draw[black, fill=black] (-1.8,-.085) circle (0.4pt);
\draw[black, fill=black] (.3,-.085) circle (0.4pt);
\draw[black, fill=black] (0.35,-.1) circle (0.4pt);
\draw[black, fill=black] (0.4,-.12) circle (0.4pt);

\draw (-1.6,.1) node {$\scriptstyle 1$};
\draw (-1.5,.55) node {$\scriptstyle 2$};
\draw [fill=black](-1.5,0) circle (.9pt);
\draw [line width = 1pt](-1.5,0) -- (-1.6,-.2);
\draw [fill=black](0,0) circle (.9pt);
\draw [line width = 1pt](0,0) -- (.1,-.2);
\draw [line width = 1pt](0,0) -- (-.1,-.2);


\draw [gray,dashed] (-1.5,0) .. controls (-1.75,1) and (.25,1) .. (0,0);
\draw [line width = 1pt](-1.5,0) .. controls (-1.55,.25) and (-1.5,.4) .. (-1.4,.5);
\draw [white, fill=white](-.75,.745) circle (4pt);
\draw [fill=gray, gray](-.83,.745) circle (.2pt);
\draw [fill=gray, gray](-.75,.745) circle (.2pt);
\draw [fill=gray, gray](-.67,.745) circle (.2pt);

\draw [fill=black, black](-1.4,.5) circle (.9pt);
\draw [fill=gray, gray](-1,.72) circle (.9pt);
\draw [fill=gray, gray](-.5,.72) circle (.9pt);
\draw [fill=gray, gray](-0.1,.5) circle (.9pt);


\draw [line width = 1pt](-1.75,-.07) .. controls (-1.25,.1) and (-.25,.1) .. (.25,-.07);
\draw [white, fill=white](-.75,.06) circle (4pt);
\draw [fill=black](-.83,.06) circle (.4pt);
\draw [fill=black](-.75,.06) circle (.4pt);
\draw [fill=black](-.67,.06) circle (.4pt);
\draw [fill=black](-1.25,.028) circle (.9pt);
\draw [fill=black](-1,.048) circle (.9pt);

\draw [fill=black](-.5,.048) circle (.9pt);
\draw [fill=black](-.25,.028) circle (.9pt);

\draw(-2.15,.3) node {$(G^\flat)^*$};

\draw [white] (1,-.3) circle (.5pt);
\draw [white] (-2.5,-.3) circle (.5pt);
\draw [white] (1,.85) circle (.5pt);
\draw [white] (-2.5,.85) circle (.5pt);

\end{tikzpicture}

\begin{tikzpicture}[line cap=round,line join=round, scale=1.8]

\draw[gray, fill=gray] (-1.9,-.12) circle (0.2pt);
\draw[gray, fill=gray] (-1.85,-.1) circle (0.2pt);
\draw[gray, fill=gray] (-1.8,-.085) circle (0.2pt);
\draw[gray, fill=gray] (.3,-.085) circle (0.2pt);
\draw[gray, fill=gray] (0.35,-.1) circle (0.2pt);
\draw[gray, fill=gray] (0.4,-.12) circle (0.2pt);

\draw (-1.6,.1) node {$\scriptstyle 1$};
\draw (-1.5,.55) node {$\scriptstyle 2$};
\draw [fill=black](-1.5,0) circle (.9pt);
\draw [gray](-1.5,0) -- (-1.6,-.2);
\draw [fill=black](0,0) circle (.9pt);
\draw [gray] (0,0) -- (.1,-.2);
\draw [gray] (0,0) -- (-.1,-.2);


\draw [line width = 1pt](-1.5,0) .. controls (-1.75,1) and (.25,1) .. (0,0);
\draw [white, fill=white](-.75,.745) circle (4pt);
\draw [fill=black](-.83,.745) circle (.4pt);
\draw [fill=black](-.75,.745) circle (.4pt);
\draw [fill=black](-.67,.745) circle (.4pt);

\draw [fill=black](-1.4,.5) circle (.9pt);
\draw [fill=black](-1,.72) circle (.9pt);

\draw [fill=black](-.5,.72) circle (.9pt);

\draw [fill=black](-0.1,.5) circle (.9pt);


\draw [gray](-1.75,-.07) .. controls (-1.25,.1) and (-.25,.1) .. (.25,-.07);

\draw [line width = 1pt](-1.5,0) .. controls (-1.1,.07) and (-.4,.07) .. (0,0);

\draw [fill=black](-1.5,0) circle (.9pt);
\draw [fill=black](0,0) circle (.9pt);

\draw [white, fill=white](-.75,.06) circle (4pt);
\draw [fill=black](-.83,.06) circle (.4pt);
\draw [fill=black](-.75,.06) circle (.4pt);
\draw [fill=black](-.67,.06) circle (.4pt);
\draw [fill=black](-1.25,.028) circle (.9pt);
\draw [fill=black](-1,.048) circle (.9pt);

\draw [fill=black](-.5,.048) circle (.9pt);
\draw [fill=black](-.25,.028) circle (.9pt);

\draw(-2.15,.3) node {$\C$};

\draw [white] (1,-.3) circle (.5pt);
\draw [white] (-2.5,-.3) circle (.5pt);
\draw [white] (1,.85) circle (.5pt);
\draw [white] (-2.5,.85) circle (.5pt);

\end{tikzpicture}

\begin{tikzpicture}[line cap=round,line join=round, scale=1.8]

\draw[black, fill=black] (-1.9,-.12) circle (0.4pt);
\draw[black, fill=black] (-1.85,-.1) circle (0.4pt);
\draw[black, fill=black] (-1.8,-.085) circle (0.4pt);
\draw[black, fill=black] (.3,-.085) circle (0.4pt);
\draw[black, fill=black] (0.35,-.1) circle (0.4pt);
\draw[black, fill=black] (0.4,-.12) circle (0.4pt);

\draw (-1.6,.1) node {$\scriptstyle 1$};
\draw (-1.5,.55) node {$\scriptstyle 2$};
\draw [fill=black](-1.5,0) circle (.9pt);
\draw [line width = 1pt](-1.5,0) -- (-1.6,-.2);
\draw [fill=black](0,0) circle (.9pt);
\draw [line width = 1pt](0,0) -- (.1,-.2);
\draw [line width = 1pt](0,0) -- (-.1,-.2);


\draw [line width = 1pt](-1.5,0) .. controls (-1.75,1) and (.25,1) .. (0,0);
\draw [white, fill=white](-.75,.745) circle (4pt);
\draw [fill=black](-.83,.745) circle (.4pt);
\draw [fill=black](-.75,.745) circle (.4pt);
\draw [fill=black](-.67,.745) circle (.4pt);

\draw [fill=black](-1.4,.5) circle (.9pt);
\draw [fill=black](-1,.72) circle (.9pt);

\draw [fill=black](-.5,.72) circle (.9pt);

\draw [fill=black](-0.1,.5) circle (.9pt);


\draw [line width = 1pt](-1.75,-.07) .. controls (-1.25,.1) and (-.25,.1) .. (.25,-.07);
\draw [line width = 2pt, white](-1.5,0) .. controls (-1.1,.07) and (-.4,.07) .. (0,0);
\draw [gray](-1.5,0) .. controls (-1.1,.07) and (-.4,.07) .. (0,0);

\draw [fill=black](-1.5,0) circle (.9pt);
\draw [fill=black](0,0) circle (.9pt);

\draw [white, fill=white](-.75,.06) circle (4pt);
\draw [gray,fill=gray](-.83,.06) circle (.2pt);
\draw [gray,fill=gray](-.75,.06) circle (.2pt);
\draw [gray,fill=gray](-.67,.06) circle (.2pt);
\draw [gray,fill=gray](-1.25,.028) circle (1pt);
\draw [gray,fill=gray](-1,.048) circle (.9pt);
\draw [gray,fill=gray](-.5,.048) circle (.9pt);
\draw [gray,fill=gray](-.25,.028) circle (.9pt);

\draw(-2.15,.3) node {$G-\I^\circ$};

\draw [white] (1,-.3) circle (.5pt);
\draw [white] (-2.5,-.3) circle (.5pt);
\draw [white] (1,.85) circle (.5pt);
\draw [white] (-2.5,.85) circle (.5pt);

\end{tikzpicture}
\end{center}
\caption{}
\label{fig: the 3 graphs in the decomposition} 
\end{figure}
By Proposition~\ref{prop: additivity of reg on leaves}, 
$\reg (G^\flat)^* = \reg G^\flat + (q-2)$.
Since $\C$ is an even cycle, 
$$
\ts \reg \C = \frac{\ell(\P)}{2}(q-2) + \frac{\ell(\I)}{2}(q-2)-(q-2).
$$
Finally, since $G^\flat$ satisfies the bipartite ear modification hypothesis, 
$$
\ts \reg (G-\I^\circ) = \reg G^\flat -\frac{\ell(\I)}{2}(q-2) + \frac{\ell(\P)}{2}(q-2).
$$
Let us write $\t^\nu = \t^\alpha \t^\beta \t^\gamma$ 
for some $\alpha,\beta,\gamma \in \NN^{E_{G}}$ satisfying  
$$
\t^\alpha \in K[E_{\P}],\quad \t^\beta \in K[E_{\I}]\quad \text{and}\quad \t^\gamma \in K[E_{G^\flat- \I^\circ}].
$$
Suppose that 
\begin{equation}\label{eq: L322}
\renewcommand{\arraystretch}{1.5}
\left \{
\begin{array}{l}
\deg(\t^\alpha) + \deg(\t^\beta) \leq \reg \C \\
\deg(\t^\beta) + \deg(\t^\gamma) \leq \reg (G^\flat)^* \\
\deg(\t^\alpha) + \deg(\t^\gamma) \leq \reg (G-\I^\circ) \\
\end{array}
\right .
\end{equation}
Then,
$$
\textstyle \deg(\t^\nu) \leq \frac12 \big [\reg \C + \reg (G^\flat)^* + \reg (G-\I^\circ) \big ] =   
\reg G^\flat +   \frac{\ell(\P)}{2} (q-2),
$$
which is in contradiction with \eqref{eq: L296}. Hence the opposite inequality of one of the inequalities in \eqref{eq: L322} must hold.
For instance if it is the opposite of the first inequality, then, by Proposition~\ref{prop: computing reg by reducing to Artinian quotient},
there exists $\t^\mu \in K[E_\C]$, divisible by $t_{12}$, such that 
$$
\t^\alpha \t^\beta - \t^\mu \in I_q(\C) \subset I_q(G)
$$
which implies that 
$$
\t^\nu - \t^\mu\t^\gamma \in I_q(G),
$$
as desired. We argue similarly for the other two cases. This proves \eqref{eq: L1824} and concludes the proof of the theorem.
\end{proof}

The next two examples show that the assumptions of 
Theorem~\ref{thm: adding a path to the endpoints of path with degree 2 inner vertices}
are strictly necessary.

\begin{exam}\label{exam: ear is needed}
Let $G^\flat$ be the graph of Figure~\ref{fig: contraexample to adding path}. 
This graph decomposes into a
cycle of length six and two cycles of length four. By Proposition~\ref{prop: reg bound from graph dec}, $\reg G^\flat \leq 4(q-2)$. On the other hand,
the set $V=\set{2,4,6,8}$ is an independent set for which $G-V$ has a nonempty edge set. 
Hence, by Proposition~\ref{prop: independent sets and regularity}, $\reg G^\flat \geq 4(q-2)$. We conclude that 
$\reg G^\flat = 4(q-2)$.
\begin{figure}[ht]
\begin{center}
\begin{tikzpicture}[line cap=round,line join=round, scale=1.5]

\draw (0,0) circle (20pt);
\draw [fill=black] (.7,0) node [anchor=west] {$\scriptstyle   2$} circle (1pt);
\draw [fill=black] (.37,0.6) node [anchor=south] {$\scriptstyle  1$} circle (1pt);
\draw [fill=black] (-.37,0.6) node [anchor=south] {$\scriptstyle   5$} circle (1pt);

\draw [fill=black] (-.7,0) node [anchor=east] {$\scriptstyle   8$} circle (1pt);
\draw [fill=black] (-.37,-0.6) node [anchor=north] {$\scriptstyle   7$} circle (1pt);
\draw [fill=black] (.37,-0.6) node [anchor=north] {$\scriptstyle  3$} circle (1pt);

\draw (-0.37,.6) .. controls (-1.5,.5) and (-1.5,-.5) ..(-0.37,-.6);
\draw [fill=black] (-1.215,.0) node [anchor=east] {$\scriptstyle   6$} circle (1pt);

\draw (0.37,.6) .. controls (1.5,.5) and (1.5,-.5) ..(0.37,-.6);
\draw [fill=black] (1.215,.0) node [anchor=west] {$\scriptstyle   4$} circle (1pt);

\end{tikzpicture}
\end{center}
\caption{}
\label{fig: contraexample to adding path}
\end{figure}
Let $G$ be the graph obtained by adding the 
edge $\set{2,8}$ to $G^\flat$. Unlike $G^\flat$, the graph $G$ has now spanning cycle (of length $8$). 
By Proposition~\ref{prop: bound from a spanning subgraph}, we get $\reg G \leq 3(q-2)$. By a similar argument as above, 
one can show that indeed $\reg G = 3(q-2)$. Thus the conclusion of Theorem~\ref{thm: adding a path to the endpoints of path with degree 2 inner vertices},
which in this case would state that $\reg G^\flat = \reg G$,
does not hold. This is because the hypothesis that, besides the edge $\set{2,8}$, 
there should be another ear in $G$ with end-vertices $2$ and $8$, is not satisfied.
\end{exam}

\begin{exam}\label{exam: bipartite is needed}
Consider the graph, $G$, illustrated in Figure~\ref{fig: contraexample to adding path in non-bipartite}.
$G$ is a non-bipartite parallel composition of paths; two of length two and a path 
of length three. 
\begin{figure}[ht]
\begin{center}
\begin{tikzpicture}[line cap=round,line join=round, scale=1.5]

\draw[fill=black] (-.33,0) circle (1pt);
\draw[fill=black] (.33,0) circle (1pt);
\draw[fill=black] (1,0) circle (1pt);
\draw[fill=black] (-1,0) circle (1pt);
\draw (-1,0)--(1,0);

\draw (-1,0) .. controls (-.5,.7) and (.5,.7) .. (1,0);
\draw (-1,0) .. controls (-.5,-.7) and (.5,-.7) .. (1,0);

\draw[fill=black] (0,.52) circle (1pt);
\draw[fill=black] (0,-.52) circle (1pt);

\end{tikzpicture}
\end{center}
\caption{}
\label{fig: contraexample to adding path in non-bipartite}
\end{figure}
According to \cite[Theorem~1.2]{MaNeVPVi},
$\reg G = 4(q-2)$. Consider $G^\flat\subset G$ the subgraph given by the parallel composition of one of the paths 
of length two and the path of length three. Then, $G^\flat$ is a cycle of length $5$ and, accordingly, 
$\reg G^\flat = 4(q-2)$. If we take $G^\flat$ to be the parallel composition of the two paths of length two. Then 
$\reg G^\flat = q-2$. In  both cases, the conclusion of Theorem~\ref{thm: adding a path to the endpoints of path with degree 2 inner vertices}
does not hold. 
\end{exam}

\section{Nested Ear Decompositions}
\label{sec: nested ear decomposition}

The goal of this section is to give a formula for the Castelnuovo--Mumford regularity of a graph endowed with
a special decomposition into paths. An ear decomposition of a graph
consists of a collection of $r>0$ subgraphs
$$
\P_0,\P_1,\dots,\P_r,
$$ 
the edge sets of which form a partition of $E_G$, such that 
$\P_0$ is a vertex and, for all $1\leq i \leq r$, $\P_i$ 
is a path with end-vertices in $\P_0\cup \cdots \cup \P_{i-1}$ while \emph{none} 
of its inner vertices belong to $\P_0\cup \cdots \cup \P_{i-1}$.
The paths $\P_1,\dots,\P_r$ are called ears of the decomposition of $G$. We note that $\P_i$ is not necessarily 
an ear of $G$, according to Definition~\ref{def: ear}, as its inner vertices may become end-vertices of the following ears.
An ear decomposition is called \emph{open} if all of paths $\P_2,\dots,\P_r$ have distinct end-vertices.
It is well known that a graph is $2$-vertex-connected if and only if it has 
an open ear decomposition (Whitney's Theorem). More generally, a graph is $2$-edge-connected 
if and only if has an ear decomposition.

\begin{definition}\label{def: nested decompositions of a graph}
Let $\P_0$, $\P_1,\dots,\P_r$ be an ear decomposition of a graph, $G$. 
If a path $\P_i$ has both its end-vertices in $\P_j$ 
we say that $\P_i$ is nested in $\P_j$ and define the corresponding \emph{nest interval} to be the subpath of $\P_j$ 
determined by the end-vertices of $\P_i$, if they are distinct, or, if they coincide, to be that single end-vertex. 
An ear decomposition of $G$ is \emph{nested} if, 
for all $1\leq i\leq r$, the path $\P_i$ is nested in a previous subgraph of the decomposition, $\P_j$, with 
$j<i$, and, in addition,  
if two paths $\P_i$ and $\P_l$ are nested in $\P_j$, with $j<i,l$, then either  
the correspon\-ding nest intervals in $\P_j$ have disjoint edge sets or one edge set is contained in the other.
\end{definition}

Nested ear decompositions were introduced by Eppstein in \cite{Ep92}. 
In the original definition $P_0$ is allowed to be a path and 
thus the graphs considered in \cite{Ep92} are not necessarily $2$-edge-connected.
\smallskip

The main result of this section is Theorem~\ref{thm: regularity of a nested ear decomposition},
which gives a formula for the Castelnuovo--Mumford regularity of a bipartite graph
endowed with a nested ear decomposition. In the proof of this result
we will need to  show that a graph endowed with a nested ear 
decomposition satisfies the bipartite ear modification hypothesis along a certain ear. 
However, an instance of a bipartite ear modification,
namely the one involving removing the ear and identifying its end-vertices can modify the 
ear decomposition structure by introducing \emph{pendant edges} and, thus, producing a graph $G'$ 
which may well not be $2$-edge connected. We remedy this by working on a wider class of
graphs, that of graphs endowed with a weaker form of nested ear decompositions.

\begin{definition}
A \emph{weak nested ear decomposition} of a graph is a collection of subgraphs $\P_0,\dots,\P_r$, 
with $r>0$, the edge sets of which form a partition of $E_G$, such that 
$\P_0$ is a vertex and, for every $1\leq i \leq r$, $\P_i$ is a path such that either 
\begin{enumerate}
\item both end-vertices of $\P_i$ belong to some $\P_j$, with $j<i$ and 
none of its inner vertices belongs $\P_0\cup \cdots \cup \P_{i-1}$, or
\item if $\ell(\P_i)=1$, only one of the end-vertices of $\P_i$ belongs to $\P_0\cup \cdots \cup \P_{i-1}$.
\end{enumerate}
If $\P_i$ has both its end-vertices in $\P_j$, the nest interval of $\P_i$ in $\P_j$ is defined as before. 
If $\ell(\P_i)=1$ and only one end-vertex belongs to a previous $\P_j$ then 
the nest interval is defined to be this vertex. The nesting condition of Definition~\ref{def: nested decompositions of a graph} 
is the same. 
\end{definition}

If both end-vertices of $\P_i$ belong to a previous $\P_j$, then $\P_i$ will be referred to as
an ear of the decomposition, otherwise, if $\ell(\P_i)=1$ and $\P_i$ has only one end-vertex in
a previous $\P_j$, then $\P_i$ will referred to as a \emph{pendant edge} of the decomposition. An ear
with coinciding end-vertices will also be referred to as a \emph{pending cycle} of the decomposition.
Notice that when $\ell(P_i)=1$, $\P_i$ can either be an ear (so-called trivial ear) or a pendant edge 
of the decomposition. Figure~\ref{fig: example of weak ear decomposition} shows a graph that 
can be endowed with a weak nested ear decomposition.
\begin{figure}[ht]
\vspace{-.5cm}
\begin{center}
\begin{tikzpicture}[line cap=round,line join=round, scale=3.5]

\draw[black, fill=black] (0,0) circle (0.5pt) node[anchor = south] {$\ss 1$};
\draw (0,0) -- (-.25,.1);
\draw[black, fill=black] (-.25,.1) circle (0.5pt) node[anchor = south east] {$\ss 2$};
\draw (0,0).. controls (.65,.65) and (.65,-.65) .. (0,0);

\draw (.3,.187) .. controls (.7,0.55) and (.9,0) .. (.485,-.05);
\draw[black, fill=black] (.3,.187) circle (0.5pt) node[anchor = south east] {$\ss 3$};
\draw[black, fill=black]  (.485,-.05) circle (0.5pt) node[anchor = north west] {$\ss 4$};
\draw[black, fill=black]  (.185,-.145) circle (0.5pt) node[anchor = north east] {$\ss 5$};
\draw[black, fill=black]  (.6,0.3) circle (0.5pt) node[anchor = south west] {$\ss 6$};

\draw (.68,.05).. controls (.8,0) and (1.1,0) .. (1.2,.05);
\draw[black, fill=black]  (.68,.05) circle (0.5pt) node[anchor = north] {$\ss 7$};
\draw[black, fill=black]  (1.2,.05) circle (0.5pt) node[anchor = north] {$\ss 9$};
\draw[black, fill=black]  (0.935,0.01) circle (0.5pt) node[anchor = north] {$\ss 8$};

\draw(1.372, 0.13) circle (5.3pt);

\draw[black, fill=black]  (1.3,0.3) circle (0.5pt) node[anchor = south] {$\ss 10$};
\draw[black, fill=black]  (1.435,-.045) circle (0.5pt) node[anchor = north] {$\ss 12$};
\draw[black, fill=black]  (1.54,0.21) circle (0.5pt) node[anchor = south west] {$\ss 11$};

\end{tikzpicture}
\end{center}
\vspace{-1cm}
\caption{}
\label{fig: example of weak ear decomposition}
\end{figure}
For instance, 
$\P_0=1$,
$\P_1=(1,2)$,
$\P_2 = (1,5)$,
$\P_3 = (1,3,4,5)$,
$\P_4 = (3,6,7,4)$,
$\P_5 = (7,8)$,
$\P_6 = (8,9)$,
$\P_7 = (9,10,11,12,9)$.
Another weak nested ear decomposition of this graph can be given by 
$\P_0=1$, 
$\P_1 = (1,2)$, 
$\P_2 = (1,3,4,5,1)$, etc., as in the previous decomposition. 
We note that the number of even ears and pendant edges of these 
decompositions is five, the same in both.
\smallskip

It is clear that a nested ear decomposition of a graph is a weak nested ear decomposition. 
Note that, as the example above shows, not all graphs endowed with a 
weak nested ear decomposition are $2$-edge-connected.  
\smallskip

We will prove Theorem~\ref{thm: regularity of a nested ear decomposition} 
by induction on the number of ears of the decomposition. For that, we will need the following lemma.

\begin{lemma}\label{lemma: existence of an ear with a nest interval ear}
Let $\P_0,\P_1, \dots, \P_r$ be a weak nested ear decomposition of a graph 
$G$. Then, there exists $i>0$ such that either $\P_i$ is a pendant edge of $G$, or a pendant cycle 
of $G$, or an ear of $G$ with distinct end-vertices 
such that for any $\P_k$ containing both end-vertices of $\P_i$, 
the subpath of $\P_k$ induced  by them is an ear of $G$. 
\end{lemma}

\begin{proof}
We argue by induction on $r\geq 1$. 
If $r=1$ then it suffices to take $i=1$. 
If $r>1$, consider the graph $G^\flat = \P_0\cup \cdots \cup \P_{r-1}$. 
By induction, there exist $i>0$ and $\P_i$ satisfying the conditions in the statement.
\begin{figure}[ht]
\begin{center}
\begin{tikzpicture}[line cap=round,line join=round, scale=2.5]

\draw[black, fill=black] (.2-1.7,-.02) circle (0.2pt);
\draw[black, fill=black] (.2-1.65,-.01) circle (0.2pt);
\draw[black, fill=black] (.2-1.6,0) circle (0.2pt);

\draw[black, fill=black] (.1-.2,0) circle (0.2pt);
\draw[black, fill=black] (0.15-.2,-.01) circle (0.2pt);
\draw[black, fill=black] (.2-.2,-.02) circle (0.2pt);

\draw (-1.35,0.01) .. controls (-1.25,.03) and (-.25,.03) .. (-.15,0.01);

\draw (-.75,.028) .. controls (0.25,0.2) and (-.35,.85) .. (-.85,.5);
\draw [white, fill=white](-.3,.42) circle (3pt);
\draw [fill=black](-.27,.5) circle (.2pt);
\draw [fill=black](-.23,.46) circle (.2pt);
\draw [fill=black](-.21,.41) circle (.2pt);
\draw [fill=black](-.45,.11) circle (.6pt);
\draw [fill=black](-.54,.6) circle (.6pt);

\draw[black] (-.05,.53)node {$\P_r$};

\draw (-.75,.028) -- (-.85,.5);
\draw [fill=black](-.85,.5) circle (.6pt);

\draw [fill=black](-1.25,.018) circle (.6pt);
\draw [fill=black](-.25,.018) circle (.6pt);
\draw [fill=black](-.75,.028) circle (.6pt);

\draw (-1,.25)node {$\P_i$};

\draw (-.75,-.2) node {\footnotesize{(a)}};

\draw [white] (.5,-.3) circle (.5pt);
\draw [white] (-2,-.3) circle (.5pt);
\draw [white] (.5,.85) circle (.5pt);
\draw [white] (-2,.85) circle (.5pt);

\end{tikzpicture}

\begin{tikzpicture}[line cap=round,line join=round, scale=2.5]

\draw[black, fill=black] (.2-1.7,-.02) circle (0.2pt);
\draw[black, fill=black] (.2-1.65,-.01) circle (0.2pt);
\draw[black, fill=black] (.2-1.6,0) circle (0.2pt);

\draw[black, fill=black] (.1-.2,0) circle (0.2pt);
\draw[black, fill=black] (0.15-.2,-.01) circle (0.2pt);
\draw[black, fill=black] (.2-.2,-.02) circle (0.2pt);

\draw (-1.35,0.01) .. controls (-1.25,.03) and (-.25,.03) .. (-.15,0.01);
\draw [fill=black](-1.25,.018) circle (.6pt);
\draw [fill=black](-.25,.018) circle (.6pt);
\draw [fill=black](-.75,.028) circle (.6pt);

\draw (-.75,.028) .. controls (-2,0.6) and (0.25,1) .. (-.75,.028);
\draw [white, fill=white](-.48,.38) circle (3pt);

\draw [fill=black](-.96,0.57) circle (.6pt);
\draw [fill=black](-1.03,.2) circle (.6pt);

\draw [fill=black](-.6,.19) circle (.6pt);
\draw [fill=black](-.57,.58) circle (.6pt);
\draw [fill=black](-.485,.44) circle (.2pt);
\draw [fill=black](-.49,.39) circle (.2pt);
\draw [fill=black](-.51,.34) circle (.2pt);

\draw (-1.25,.35)node {$\P_i$};

\draw (-.6,.19) .. controls (0.25,0.2) and (-.15,.85) .. (-.57,.58);
\draw [white, fill=white](-.2,.32) circle (3pt);
\draw [fill=black](-.225,.265) circle (.2pt);
\draw [fill=black](-.175,.305) circle (.2pt);
\draw [fill=black](-.13,.35) circle (.2pt);

\draw [fill=black](-.38,.21) circle (.6pt);
\draw [fill=black](-.35,.645) circle (.6pt);
\draw [fill=black](-.11,.5) circle (.6pt);

\draw[black] (0,.23)node {$\P_r$};

\draw (-.75,-.2) node {\footnotesize{(b)}};

\draw [white] (.5,-.3) circle (.5pt);
\draw [white] (-2,-.3) circle (.5pt);
\draw [white] (.5,.85) circle (.5pt);
\draw [white] (-2,.85) circle (.5pt);

\end{tikzpicture}

\begin{tikzpicture}[line cap=round,line join=round, scale=2.5]

\draw[black, fill=black] (-1.21,-.1) circle (0.2pt);
\draw[black, fill=black] (-1.187,-.142) circle (0.2pt);
\draw[black, fill=black] (-1.16,-.18) circle (0.2pt);

\draw[black, fill=black] (-.285,-.1) circle (0.2pt);
\draw[black, fill=black] (-.31,-.14) circle (0.2pt);
\draw[black, fill=black] (-.34,-.18) circle (0.2pt);

\draw (-1.25,.018) .. controls (-1.25,.03) and (-.25,.03) .. (-.25,.018);
\draw [fill=black](-1.25,.018) circle (.6pt);
\draw [fill=black](-.25,.018) circle (.6pt);
\draw [fill=black](-.95,.028) circle (.6pt);
\draw [fill=black](-.55,.028) circle (.6pt);
\draw [white, fill=white](-.75,0) circle (3pt);
\draw [fill=black](-.80,.025) circle (.2pt);
\draw [fill=black](-.75,.025) circle (.2pt);
\draw [fill=black](-.70,.025) circle (.2pt);
\draw (-1.25,.018) -- (-1.27,-.05);
\draw (-1.25,.018) -- (-1.23,-.05);
\draw (-.25,.018) -- (-.23,.-.06);
\draw (-.25,.018) -- (-.27,.-.06);

\draw (-0.96,.12)node {$\P_i$};

\draw (-1.25,.018) .. controls (-1.5,1) and (0,1) .. (-.25,.018);
\draw (-1.25,.018) .. controls (-1.25,.75) and (-0.25,.75) .. (-.25,.018);
\draw (-1.25,.018) .. controls (-1,.4) and (-.5,.4) .. (-.25,.018);

\draw [white, fill=white](-.53,.24) circle (3pt);
\draw [white, fill=white](-.42,.4) circle (3pt);
\draw [white, fill=white](-.3,.55) circle (3pt);

\draw [fill=black](-.35,.58) circle (.2pt);
\draw [fill=black](-.31,.53) circle (.2pt);
\draw [fill=black](-.28,.48) circle (.2pt);

\draw [fill=black](-.45,.47) circle (.2pt);
\draw [fill=black](-.40,.42) circle (.2pt);
\draw [fill=black](-.36,.37) circle (.2pt);

\draw [fill=black](-.58,.275) circle (.2pt);
\draw [fill=black](-.52,.255) circle (.2pt);
\draw [fill=black](-.47,.22) circle (.2pt);

\draw [fill=black](-.75,0.3) circle (.6pt);
\draw [fill=black](-.65,0.555) circle (.6pt);
\draw [fill=black](-.55,0.722) circle (.6pt);

\draw[black] (-1.15,.75)node {$\I_l$};

\draw[black] (-1.3,-.25)node {$\P_l$};

\draw (-.55,0.722) .. controls (.25,1.05) and (.5,.2) .. (-.25,.018);
\draw [white, fill=white](.1,.2) circle (3pt);
\draw [fill=black](0.06,.17) circle (.2pt);
\draw [fill=black](0.1,.21) circle (.2pt);
\draw [fill=black](0.13,.25) circle (.2pt);
\draw [fill=black](-.05,0.09) circle (.6pt);
\draw [fill=black](-.05,0.77) circle (.6pt);
\draw [fill=black](.19,0.37) circle (.6pt);

\draw[black] (0.35,.5)node {$\P_r$};

\draw (-.75,-.4) node {\footnotesize{(c)}};

\draw [white] (.5,-.3) circle (.5pt);
\draw [white] (-2,-.3) circle (.5pt);
\draw [white] (.5,.85) circle (.5pt);
\draw [white] (-2,.85) circle (.5pt);

\end{tikzpicture}
\end{center}
\vspace{-.5cm}
\caption{}
\label{fig: In the lema}
\end{figure}
If $\P_r$ is pendant edge we may take $i=r$ for $G$.
The same applies if $\P_r$ is a pending cycle. Assume then that 
$\P_r$ has distinct end-vertices. If $\P_i$ is a pendant edge of $G^\flat$ and it ceases
to be so in $G$, then the end-vertices of $\P_r$ must coincide with the end-vertices of 
$\P_i$ and the only $\P_k$ that contain these vertices are then $\P_r$ and $\P_i$ which are both ears
of $G$. (See Figure~\ref{fig: In the lema}a.)
In this case, $\P_r$ satisfies the conditions for $G$. 
If $\P_i$ is a pending cycle of $G^\flat$ which ceases to be an ear of $G$ then one 
of the end-vertices of $\P_r$ must 
be an inner vertex of $\P_i$. Then, arguing as before, we see that $\P_r$ satisfies the conditions. 
(See Figure~\ref{fig: In the lema}b.)
Finally, assume that $\P_i$ is an ear of $G^\flat$ with distinct end-vertices. 
Let $\I_1,\dots,\I_{r-1}$ be the subpaths induced by the end-vertices of $\P_i$ in the paths $\P_1,\dots,\P_{r-1}$.
For ease of notation consider $\I_j$ equal to the empty set if $\P_j$ does not contain both end-vertices of $\P_i$. 
If the end-vertices of $\P_r$ do not coincide with any 
inner vertex of the paths $\I_i,\dots,\I_{r-1}$ then $\P_i$ satisfies the conditions of the statement for $G$. 
Assume that an end-vertex of $\P_r$ is an inner vertex of $\I_l$. 
Then, as $\I_l$ is an ear of $G^\flat$, $\P_r$ has to be nested in $\P_l$. Since 
$\P_i$ is also nested in $\P_l$ the nest intervals must be nested. This means that the subpath induced by 
the end-vertices of $\P_r$ in $\P_l$ must be contained in $\I_l$. (See Figure~\ref{fig: In the lema}c.) 
Then, $\P_r$ satisfies the conditions of the statement for $G$ as the only paths 
that contain the end-vertices of $\P_r$ are then $\P_r$ and $\P_l$.
\end{proof}

\begin{theorem}\label{thm: regularity of a nested ear decomposition}
Let $\P_0,\P_1, \dots,\P_r$ be a weak nested ear decomposition of a bipartite graph, $G$. 
Let $\epsilon$ denote the number of even ears and pendant edges of the decomposition. Then, 
\begin{equation}\label{eq:1005}
\textstyle \reg G = \frac{|V_G|+\epsilon-3}{2} (q-2).
\end{equation}
\end{theorem}

\begin{proof}
We will argue by induction on $r\geq 1$. If $r=1$ then $G$ is either an even cycle 
or a single edge. In both cases $\epsilon =1$ and \eqref{eq:1005} gives 
$\reg G = \frac{|V_G|-2}{2}(q-2)$, in the case of the even cycle 
and $\reg G = 0$, in the case of the edge. Both are correct. 
Assume that \eqref{eq:1005} holds for any bipartite graph endowed 
with a weak nested ear decomposition with $r$ paths and consider $G$ a graph endowed with 
a weak nested ear decomposition $\P_0,\dots, \P_{r+1}$ with $r+1$ paths. Throughout the remainder of the proof, denote by $\epsilon$ the number of 
even ears and pendant edges of this decomposition.
By Lemma~\ref{lemma: existence of an ear with a nest interval ear},
there exist $i>0$, such that $\P_i$ is either a pendant edge, or a pendant cycle, or an ear of $G$
with distinct end-vertices 
such that for any $\P_k$ containing both end-vertices of $\P_i$, 
the subpath of $\P_k$ induced  by them is an ear of $G$. 
It is clear that in any of the cases 
\begin{equation}\label{eq: L2249}
G^\flat=\P_0\cup \cdots \cup \P_{i-1}\cup \P_{i+1}\cup \cdots \cup \P_{r+1}
\end{equation} 
is a bipartite graph endowed with a weak nested 
ear decomposition. If $\P_{i}$ is a pendant edge of $G$, then by Proposition~\ref{prop: additivity of reg on leaves} and induction,
$$
\ts \reg G = \reg G^\flat + (q-2) = \frac{|V_G|-1+(\epsilon -1)-3}{2}(q-2) + (q-2) = \frac{|V_G|+\epsilon-3}{2} (q-2).
$$
If $\P_{i}$ is pendant cycle of $G$, then 
$\ell(P_{i})$ is even and, 
by induction and Proposition~\ref{prop: adding a path to the endpoints of an edge},
$$
\ts \reg G =  \frac{|V_G|-\ell(\P_{i})+1+(\epsilon -1)-3}{2}(q-2) + \frac{\ell(\P_{i})}{2}(q-2) = \frac{|V_G|+\epsilon-3}{2} (q-2).
$$
Assume that $\P_{i}$ has distinct end-vertices and that for any $\P_k$ containing both end-vertices of $\P_i$, 
the subpath of $\P_k$ induced  by them is an ear of $G$. Denote the end-vertices
of $\P_i$ by $v$ and $w$. If $v$ and $w$ are adjacent in $G$ then 
$\ell(\P_i)$ must be odd. Accordingly, by induction and Proposition~\ref{prop: adding a path to the endpoints of an edge},
$$
\ts \reg G = \frac{|V_G|-\ell(\P_{i})+1+\epsilon-3}{2}(q-2) + \frac{\ell(\P_{i})-1}{2}(q-2) = \frac{|V_G|+\epsilon-3}{2} (q-2).
$$
Assume now that $v$ and $w$ are not adjacent in $G$ and let $j>0$ be the least positive integer 
such that $\P_i$ is nested in $\P_j$. By the minimality of $j$ one of $v$ or $w$, say $v$, belongs to $\P_j^\circ$.
Denote the nest interval of $\P_i$ in $\P_j$ by $\I$. (See top of Figure~\ref{fig: Fixing minimal P_j}.)
\begin{figure}[ht]
\begin{center}
\begin{tikzpicture}[line cap=round,line join=round, scale=2.5]

\draw[black, fill=black] (-1.7,-.055) circle (0.2pt);
\draw[black, fill=black] (-1.65,-.045) circle (0.2pt);
\draw[black, fill=black] (-1.6,-.035) circle (0.2pt);

\draw[black, fill=black] (.1,-.035) circle (0.2pt);
\draw[black, fill=black] (0.15,-.045) circle (0.2pt);
\draw[black, fill=black] (0.2,-.055) circle (0.2pt);

\draw [fill=black](-1.5,0) circle (.6pt);
\draw [fill=black](0,0) circle (.6pt);

\draw (-1.55,-.02) .. controls (-1.25,.08) and (-.25,.08) .. (.05,-.02);
\draw [white, fill=white](-.75,.06) circle (4pt);
\draw [fill=black](-.83,.06) circle (.2pt);
\draw [fill=black](-.75,.06) circle (.2pt);
\draw [fill=black](-.67,.06) circle (.2pt);
\draw [fill=black](-1.25,.028) circle (.6pt);
\draw [fill=black](-1,.048) circle (.6pt);
\draw [fill=black](-.5,.048) circle (.6pt);
\draw [fill=black](-.25,.028) circle (.6pt);

\draw (-1.25,.028) .. controls (-1.5,1) and (0,1) .. (-.25,.028);
\draw [white, fill=white](-.75,.8) circle (4pt);
\draw [fill=black](-.83,.75) circle (.2pt);
\draw [fill=black](-.75,.76) circle (.2pt);
\draw [fill=black](-.67,.75) circle (.2pt);
\draw [fill=black](-1.05,.68) circle (.6pt);
\draw [fill=black](-.45,.68) circle (.6pt);

\draw[black] (-.2,.63)node {$\P_k$};

\draw (-1.25,.028) .. controls (-1.25,.75) and (-.25,.75) .. (-.25,.028);
\draw [fill=black](-1.15,.38) circle (.6pt);
\draw [fill=black](-.35,.38) circle (.6pt);
\draw [fill=black](-.75,.57) circle (.6pt);

\draw (-1.25,.028) .. controls (-1.15,.5) and (-.35,.5) .. (-.25,.028);
\draw [fill=black](-.75,.38) circle (.6pt);

\draw [gray](-1.25,.028) .. controls (-1.25,-.5) and (-.25,-.5) .. (-.25,.028);
\draw [white, fill=white](-.75,-.4) circle (4pt);
\draw [fill=gray,gray](-.83,-.36) circle (.2pt);
\draw [fill=gray,gray](-.75,-.37) circle (.2pt);
\draw [fill=gray,gray](-.67,-.36) circle (.2pt);
\draw [fill=gray,gray](-1.1,-.26) circle (.6pt);
\draw [fill=gray,gray](-.4,-.26) circle (.6pt);

\draw[gray] (-.6,-.23)node {$\P_i$};

\draw [fill=black](-1.25,.028) circle (.6pt);
\draw [fill=black](-.25,.028) circle (.6pt);

\draw (-1.36,.1) node {$\ss v$};
\draw (-.12,.1) node {$\ss w$};

\draw (-1.6,.1)node {$\P_j$};
\draw (-.75,.16)node {$\I$};

\end{tikzpicture}

\begin{tikzpicture}[line cap=round,line join=round, scale=1.5]
\BDarrow{0}
\end{tikzpicture}

\vspace{-1.5cm}
\begin{tikzpicture}[line cap=round,line join=round, scale=2.5]

\draw[black, fill=black] (.2-1.7,-.02) circle (0.2pt);
\draw[black, fill=black] (.2-1.65,-.01) circle (0.2pt);
\draw[black, fill=black] (.2-1.6,0) circle (0.2pt);

\draw[black, fill=black] (.1-.2,0) circle (0.2pt);
\draw[black, fill=black] (0.15-.2,-.01) circle (0.2pt);
\draw[black, fill=black] (.2-.2,-.02) circle (0.2pt);

\draw (-1.35,0.01) .. controls (-1.25,.03) and (-.25,.03) .. (-.15,0.01);

\draw (-.75,.028) .. controls (-0.5,1.5) and (.75,.5) .. (-.75,.028);
\draw [white, fill=white](-.1,.75) circle (4pt);
\draw [fill=black](-.17,.8) circle (.2pt);
\draw [fill=black](-.1,.76) circle (.2pt);
\draw [fill=black](-.05,.7) circle (.2pt);
\draw [fill=black](-.22,.29) circle (.6pt);
\draw [fill=black](-.57,.55) circle (.6pt);

\draw[white] (-2.22,.5) circle (1pt);

\draw (-.75,.028) .. controls (-2,.25) and (-1.25,.75) .. (-.75,.028);
\draw [fill=black](-1.07,.33) circle (.6pt);
\draw [fill=black](-1.15,.13) circle (.6pt);
\draw [fill=black](-1.42,.37) circle (.6pt);

\draw[black] (.15,.63)node {$\P'_k$};

\draw (-.75,.028) -- (-.85,.5);
\draw [fill=black](-.85,.5) circle (.6pt);

\draw [fill=black](-1.25,.018) circle (.6pt);
\draw [fill=black](-.25,.018) circle (.6pt);
\draw [fill=black](-.75,.028) circle (.6pt);

\draw (-.75,-.1) node {$\ss v=w$};

\draw (-1.6,.1)node {$\P'_j$};

\end{tikzpicture}
\end{center}
\caption{}
\label{fig: Fixing minimal P_j}
\end{figure}

\noindent
To be able to use Theorem~\ref{thm: adding a path to the endpoints of path with degree 2 inner vertices}, 
it will now suffice to show that $G^\flat$ satisfies the bipartite ear modification hypothesis on $\I$.
If $G'$ is a bipartite ear modification of $G^\flat$ along $\I$, which does not involve the identification of the end-vertices 
of $\I$, then, as $v$ and $w$ are not adjacent, no multiple edges arise and 
the weak nested ear decomposition of $G^\flat$ induces a weak nested ear decomposition of $G'$ in which the only 
change is in the length of $\P_j$, which, nevertheless, remains of the same parity. By induction, we can use 
\eqref{eq:1005} on both $G^\flat$ and $G'$ and the only change will be on the cardinality of the sets of
vertices of these graphs. It follows that 
$$
\ts \reg G' = \reg G^\flat + \frac{|V_{G^\flat}|-|V_{G'}|}{2}(q-2),
$$
which is condition \eqref{eq: the hypothesis} of 
the bipartite ear modification hypothesis. 
\smallskip

\noindent
Suppose now that $\ell(\I)$ is even and that $G'$ is obtained by identifying
the end-vertices of $\I$ in $G^\flat - \I^\circ$ and removing the multiple edges created. 
For $k\not =i$, let 
$\P'_k\subset G'$ denote the graph obtained by identifying $v$ with 
$w$ in $\P_k-\I^\circ$ and removing the multiple edges created. (See Figure~\ref{fig: Fixing minimal P_j}.) 
We note that since $\I$ is an ear of $G$, for 
$k\not = j$, the graph $\P'_k$ is obtained by simply identifying $v$ and $w$ in 
$\P_k$ and removing all multiple edges created. $\P_k$ is isomorphic to $\P'_k$ if one of $v$ or $w$
does not belong to $\P_k$. If both $v$ and $w$ belong to $\P_k$ then, by the minimality of $j$, we 
must have $j<k$. But then none of $v$ or $w$ can be an inner vertex of $\P_k$. Accordingly, they must
coincide with the end-vertices of $\P_k$. (See top of Figure~\ref{fig: Fixing minimal P_j}.)
Then, $\I\cup \P_k$ is a cycle of $G$ and, since it must be of even length, we deduce that 
$\ell(\P_k)$ is also even. In this situation $\P'_k$ is either a pending (even) cycle of $G'$,
if $\ell(\P_k)>2$, or a pending edge, if $\ell(\P_k)=2$. (See the bottom of Figure~\ref{fig: Fixing minimal P_j}.)
As for $\P'_j$, it may be a single edge if $\ell(\P_j)-\ell(\I)=1$ or if $\ell(\P_j)-\ell(\I)=2$ and $\P_j$ 
has coinciding end-vertices.
\smallskip

\noindent
It is clear that
\begin{equation}\label{eq: L2311}
G' = \P_0\cup \P'_1 \cup \cdots\cup \P'_{i-1} \cup \P'_{i+1} \cup\cdots \cup \P'_{r+1}, 
\end{equation}
as any edge in $G'$ comes from an edge in some path $\P_k$. If
\eqref{eq: L2311} does not induce a partition of the edge set of $G'$, 
there exist a vertex $u$ and two edges $\set{u,v}$, $\set{u,w}$ belonging to different paths, which become 
the same edges after the identification of $v$ with $w$.
Consider the least $k$ for which the path $\P_k$ contains both vertices $u,v$ and the least $l$ for which 
$\P_l$ contains $u,w$. We claim that $\P_j=\P_k$, and, consequently, $u$ must belong to $\P_j$.
\smallskip

\noindent
Assume, to the contrary, that $j\not = k$. Since $v\in \P_j^\circ$ we have $j<k$ and  
then $v$ is an end-vertex of $\P_k$. If $u$ in an inner vertex of $\P_k$ then we must have $k\leq l$. 
If $u$ is an end-vertex of $\P_k$ then, by the minimality of $k$, $\P_k$
has to be a pending edge and we get the same conclusion, $k\leq l$. Assume that $k=l$. Then $j<k=l$ implies that
$\P_k=\P_l$ is a path with end-vertices $v$ and $w$, containing $u$ as an interior point. 
However if $\set{u,v}$ and $\set{u,w}$ belong to different paths then 
the degree of $u$ is not two, which contradicts the assumption on $v$ and $w$, stating that 
these vertices induce on any path of the weak decomposition a subpath the inner vertices of which have degree 
two. Hence we must have $k<l$. Then $j<k<l$ implies that $\P_l$ is the edge $\set{u,w}$. 
Since both its end-vertices belong to earlier paths, 
this contradicts the minimality of $l$. 
\smallskip

\noindent
We have proved that $\P_j=\P_k$ and, in particular, that $u\in \P_j$. Resetting notation,   
Let now $\P_k$ and $\P_l$ be \emph{any} two (distinct) paths containing the edges
$\set{u,v}$ and $\set{u,w}$, respectively. We claim that
either $\P_k$ is a non-pending odd path of the
weak ear decomposition of $G^\flat$ and $\P'_k$ is an edge or
$\P_l$ is a non-pending odd path of the
weak ear decomposition of $G^\flat$ and $\P'_l$ is an edge.
\smallskip 

\noindent
To see this, we start by noting that, since $v\in\P_j^\circ$, we have $j\leq k$. If $j<k$, then as $u$ and $v$ belong to 
$\P_j$ they must be end-vertices of $\P_k$. In this case, $\P_k=\set{u,v}$, which is a non-pending odd path
and $\P'_k$ an edge. If $j=k$ then there are two sub-cases.
Either $j=k<l$ and then $\P_l=\set{u,w}$, which is a non-pending odd ear with $\P'_l$ an edge  
(see Figure~\ref{fig: the paths}a), or we have $l<j=k$ and then $u$ and $w$ must be the end-vertices of $\P_j=\P_k$. This implies 
that $\ell(P_j)=\ell(\I)+1$, which is a odd integer, and that $\P_j'$ is an edge. (See Figure~\ref{fig: the paths}b.)

\begin{figure}[ht]
\begin{center}
\begin{tikzpicture}[line cap=round,line join=round, scale=2.5]

\draw (-0.6,0) .. controls (-0.25,.08) and (.25,.08) .. (.6,0);

\draw [fill=black](-.63,-.01) circle (.2pt);
\draw [fill=black](-.66,-.02) circle (.2pt);
\draw [fill=black](-.69,-.03) circle (.2pt);
\draw [fill=black](.63,-.01) circle (.2pt);
\draw [fill=black](.66,-.02) circle (.2pt);
\draw [fill=black](.69,-.03) circle (.2pt);

\draw [fill=black](-.50,.022) circle (.6pt);
\draw [fill=black](-.30,.046) circle (.6pt);
\draw [fill=black](-.10,.058) circle (.6pt);
\draw [fill=black](.50,.022) circle (.6pt);

\draw [white, fill=white](.2,.06) circle (3pt);
\draw [fill=black](.15,.056) circle (.2pt);
\draw [fill=black](.20,.055) circle (.2pt);
\draw [fill=black](.25,.052) circle (.2pt);

\draw (-0.5,0) .. controls (-0.4,.4) and (.4,.4) .. (.5,0);

\draw (-.5,-.07) node {$\ss u$};
\draw (-0.3,-.05) node {$\ss v$};
\draw (.5,-.07) node {$\ss w$};

\draw (1,-.1) node {$\P_j=\P_k$};
\draw (.4,.3)node {$\P_l$};

\draw (0,-.3) node {\footnotesize{(a)}};


\draw[white](-.8,-.2) circle (.1pt);
\draw[white](1.25,-.2) circle (.1pt);

\end{tikzpicture}
\begin{tikzpicture}[line cap=round,line join=round, scale=2.5]

\draw (-0.6,0) .. controls (-0.25,.08) and (.25,.08) .. (.6,0);

\draw [fill=black](-.63,-.01) circle (.2pt);
\draw [fill=black](-.66,-.02) circle (.2pt);
\draw [fill=black](-.69,-.03) circle (.2pt);
\draw [fill=black](.63,-.01) circle (.2pt);
\draw [fill=black](.66,-.02) circle (.2pt);
\draw [fill=black](.69,-.03) circle (.2pt);

\draw [fill=black](-.30,.046) circle (.6pt);
\draw [fill=black](.30,.046) circle (.6pt);

\draw (-0.3,.046) .. controls (-0.6,.7) and (.6,.7) .. (.3,.046);
\draw [fill=black](-.345,.3) circle (.6pt);
\draw [fill=black](.345,.3) circle (.6pt);
\draw [fill=black](-.15,.51) circle (.6pt);
\draw [white, fill=white](.15,.5) circle (2.5pt);
\draw [fill=black](.2,.485) circle (.2pt);
\draw [fill=black](.155,.505) circle (.2pt);
\draw [fill=black](.11,.52) circle (.2pt);

\draw (-.43,.3) node {$\ss v$};
\draw (-.3,-.05) node {$\ss u$};
\draw (.3,-.05) node {$\ss w$};

\draw  (-.55,.5) node {$\P_k=\P_j$};
\draw (-.8,-.1)node {$\P_l$};

\draw (0,-.3) node {\footnotesize{(b)}};

\draw[white](-1.05,-.2) circle (.1pt);
\draw[white](1.0,-.2) circle (.1pt);

\end{tikzpicture}
\end{center}
\vspace{-.35cm}
\caption{}
\label{fig: the paths}
\end{figure}

\noindent
We conclude that, for \eqref{eq: L2311} to induce a partition of the edge set of $G'$ it suffices 
to remove appropriately from \eqref{eq: L2311} the paths that consist of a single repeated edge coming
from paths as described above. 
We note that, in this situation, the number of even ears and pending edges, after removing
all repeated edges in $\eqref{eq: L2311}$, coincides with the number of even ears and pending edges 
of the weak nested decomposition of $G^\flat$ given by \eqref{eq: L2249}.
\smallskip

\noindent
Let us now prove that, after the exclusion from \eqref{eq: L2311} of the repeated edges, 
we obtain a weak nested ear decomposition of $G'$. 
Since $\I$ is an ear of $G$, it is clear that the end-vertices of $\P'_k$
belong to $\P'_l$, for some $l<k$.  
An inner vertex of $\P'_k$ always comes from an inner vertex of $\P_k$. 
If, following the bipartite ear modification, such vertex is identified with a vertex of a previous $\P'_l$, 
then the two vertices in question must $v$ and $w$. We deduce that $\P_k = \P_j$.
If $w$ is also an inner vertex of $\P_j$ then $v=w$ in $\P'_j$ cannot belong to any earlier $\P'_l$. If $w$ is an end-vertex of $\P_j$ then 
$v=w$ becomes an end-vertex of $\P'_j$.
As for the nesting condition,  let $\P'_k$ and $\P'_{l}$ be two paths nested in $\P'_s$, with $s<k,l$. 
We may assume that none of $\P'_k$ or $\P'_{l}$ is a pending edge 
or cycle for otherwise there will be nothing to show. Let $v_1\not = v_2$ be the end-vertices of $\P'_k$ and 
$w_1\not =w_2$ those of $\P'_{l}$. If the nest intervals of $\P'_k$ and $\P'_{l}$ in $\P'_s$ have inner 
vertices in common and are not nested, then, in the ordering of vertices of $\P'_s$, 
we must have, without loss of generality, 
$v_1<w_1<v_2<w_2$. This is impossible before the identification of $v$ with $w$, since we are starting from  
a weak nested ear decomposition of $G^\flat$. 
Hence both $v$ and $w$ belong to $\P_s$ and
one of $w_1$ or $v_2$ must be the vertex obtained by identifying
$v$ with $w$. Assume, without loss of generality, that this vertex is $w_1$. 
Then $v_1,v_2$ come from $\P_s$ unchanged and one of them is an inner 
vertex of the subpath induced by $v$ and $w$ in $\P_s$, but this is impossible since,
by assumption, $v$ and $w$ induce on $\P_s$ a subpath whose inner vertices have degree 
two in $G$.
\smallskip

\noindent
Having established that, after removing all redundant edges, \eqref{eq: L2311} 
induces a weak nested ear decomposition of $G'$ we can now use induction to compute
its regularity. Since the fact that $\ell(\I)$ is even implies that $\ell(\P_i)$ is even, we deduce 
that the number of even ears and pending edges of the weak decomposition of $G^\flat$ given 
in \eqref{eq: L2249} is $\epsilon -1$. Accordingly, 
$$
\ts \reg G' = \frac{|V_{G'}|+(\epsilon -1)-3}{2}(q-2) = 
\reg G^\flat + \frac{|V_{G'}|-|V_{G^\flat}|}{2}(q-2),
$$
and therefore, $G^\flat$ satisfies the bipartite ear modification assumption along $\I$.
This finishes the proof of the theorem.
\end{proof}

The number of even length paths in an ear decomposition of a graph is not necessarily constant, even if we  
restrict to bipartite graphs. Take, for example, the graph obtained from the graph in Figure~\ref{fig: contraexample to adding path}
by adding the edge $\set{2,8}$. As a first ear decomposition, consider the one obtained by starting from vertex $1$ and adding
consecutively the paths $(1,2,3,7,8,5,1)$, $(1,4,3)$, $(5,6,7)$ and $(2,8)$. This 
decomposition has three even length ears. Alternatively, consider the ear decomposition starting
from vertex $1$, using first the Hamiltonian cycle: $(1,2,8,5,6,7,3,4,1)$, followed by $(1,5)$, $(7,8)$ and $(2,3)$. 
This decomposition has only one even length ear.
\medskip

The following purely combinatorial result follows easily from Theorem~\ref{thm: regularity of a nested ear decomposition}.

\begin{cor}\label{cor: number of even length paths is constant}
Let $G$ be a bipartite graph endowed with a weak nested ear decomposition. Then the number of even ear and pendant 
edges of any weak nested ear decomposition of $G$ remains constant.
\end{cor}

We will finish by giving another application of Theorem~\ref{thm: regularity of a nested ear decomposition}.
As mentioned in Section~\ref{sec: prelim}, it follows from Proposition~\ref{prop: independent sets and regularity} that 
$$
\reg G \geq (\alpha(G)-1)(q-2),
$$  
where $\alpha(G)$ denotes the independence number of $G$. This bound is not sharp if $G$ is not bipartite or $2$-connected, but 
equality does hold if $G$ is an even cycle, or a complete bipartite graph, or a bipartite
parallel composition of paths (see \cite{MaNeVPVi}), among many other examples. 
This could suggest that for a bipartite $2$-connected 
graph the Castelnuovo--Mumford regularity of a graph is closely related with $\alpha(G)$. 
In the following example we want to show that this is not the case.

\begin{exam}\label{exam: taino sun}
Fix an even positive integer $k$. Consider the graph $G$, in Figure~\ref{fig: Taino Sun}, below, obtained from a cycle of length $3k$, 
by attaching $k$ ears of length two at the pairs of vertices $3i-2$ and $3i$, for each $i=1,\dots,k$.
\begin{figure}[ht]
\begin{center}
\begin{tikzpicture}[line cap=round,line join=round, scale=1.75]

\draw (0,0) circle (28.5pt);

\foreach \k in {0,1,2,3,4,5,6,7,8,9,10,11}
{
\draw[fill=black] ({cos(\k * 30)},{sin(\k * 30)}) circle (.8pt);
}

\draw ({cos(4* 30)},{sin(4 * 30)}) .. controls ({1.5*cos(3.5* 30)},{1.5*sin(3.5 * 30)}) and ({1.5*cos(2.5* 30)},{1.5*sin(2.5 * 30)}) .. ({cos(2* 30)},{sin(2 * 30)});
\draw[fill=black] ({1.3*cos(3 * 30)},{1.3*sin(3 * 30)}) circle (.8pt);

\draw ({cos(-4* 30)},{sin(-4 * 30)}) .. controls ({1.5*cos(-3.5* 30)},{1.5*sin(-3.5 * 30)}) and ({1.5*cos(-2.5* 30)},{1.5*sin(-2.5 * 30)}) .. ({cos(-2* 30)},{sin(-2 * 30)});
\draw[fill=black] ({1.3*cos(0 * 30)},{1.3*sin(0 * 30)}) circle (.8pt);

\draw ({cos(-30)},{sin(-30)}) .. controls ({1.5*cos(-.5* 30)},{1.5*sin(-.5 * 30)}) and ({1.5*cos(.5* 30)},{1.5*sin(.5 * 30)}) .. ({cos(30)},{sin(30)});
\draw[fill=black] ({1.3*cos(-3 * 30)},{1.3*sin(-3 * 30)}) circle (.8pt);

\draw ({cos(5* 30)},{sin(5 * 30)}) .. controls ({1.5*cos(5.5* 30)},{1.5*sin(5.5 * 30)}) and ({1.5*cos(6.5* 30)},{1.5*sin(6.5 * 30)}) .. ({cos(7* 30)},{sin(7 * 30)});
\draw[fill=black] ({1.3*cos(6 * 30)},{1.3*sin(6 * 30)}) circle (.8pt);

\draw [white, fill=white](0,-1) circle (4pt);
\draw [fill=black](0,-1.008) circle (.23pt);
\draw [fill=black](.07,-1) circle (.23pt);
\draw [fill=black](-.07,-1) circle (.23pt);

\draw [white, fill=white](0,-1.27) circle (4pt);
\draw [fill=black](0,-1.308) circle (.23pt);
\draw [fill=black](.07,-1.3) circle (.23pt);
\draw [fill=black](-.07,-1.3) circle (.23pt);

\draw ({.85*cos(0 * 30)},{.85*sin(0 * 30)}) node {$\ss 5$};
\draw ({.85*cos(1 * 30)},{.85*sin(1 * 30)}) node {$\ss 4$};
\draw ({.85*cos(2 * 30)},{.85*sin(2 * 30)}) node {$\ss 3$};
\draw ({.85*cos(3 * 30)},{.85*sin(3 * 30)}) node {$\ss 2$};
\draw ({.85*cos(4 * 30)},{.85*sin(4 * 30)}) node {$\ss 1$};
\draw ({.83*cos(5 * 30)},{.83*sin(5 * 30)}) node {$\ss 3k$};
\draw ({.72*cos(6 * 30)},{.72*sin(6 * 30)}) node {$\ss 3k-1$};
\draw ({.72*cos(7 * 30)},{.85*sin(7 * 30)}) node {$\ss 3k-2$};

\draw ({.85*cos(11 * 30)},{.85*sin(11 * 30)}) node {$\ss 6$};

\draw (0,1.45) node {$\ss 3k+1$};
\draw (1.6,0) node {$\ss 3k+2$};
\draw (-1.5,0) node {$\ss 4k$};


\draw[white](-1.85,-.2) circle (.1pt);
\draw[white](1.85,-.2) circle (.1pt);

\end{tikzpicture}
\end{center}
\vspace{-.35cm}
\caption{}
\label{fig: Taino Sun}
\end{figure}
This graph is endowed with a nested ear decomposition with $k+1$ even ears. According to 
Theorem~\ref{thm: regularity of a nested ear decomposition}, 
$$
\ts \reg G = \frac{4k+k+1-3}{2}(q-2) = (\frac {5k}{2}-1)(q-2).
$$
On the other hand, as any set of independent vertices must have at most two elements in the $k$ cycles of length $4$ created by the addition 
of the $k$ ears and the vertex sets of these cycles cover $V_G$, 
we deduce that $\alpha(G)$ is $2k$. 
In conclusion, this example shows that, indeed, 
$$
\ts \reg G - (\alpha(G)-1)(q-2) = \frac{k}{2}(q-2)
$$ 
can be arbitrarily large.
\end{exam}

\end{document}